\documentclass[12pt]{amsart}
\usepackage{latexsym,amsmath,amssymb,epsfig,amsthm}
\usepackage{graphicx}
\usepackage{tikz}
\usepackage[enableskew,vcentermath]{youngtab}
\usepackage{hyperref}
\hypersetup{colorlinks=false}
\input epsf
\newtheorem{theorem}{Theorem}[section]
\newtheorem{proposition}[theorem]{Proposition}
\newtheorem{corollary}[theorem]{Corollary}

\newtheorem{definition}[theorem]{Definition}
\newtheorem{example}[theorem]{Example}
\newtheorem{lemma}[theorem]{Lemma}
\newtheorem{remark}[theorem]{Remark}

\newcommand{\asc}{{\rm asc}}

\newcommand{\des}{{\rm des}}

\newcommand{\esd}{{\rm esd}}

\newcommand{\inte}{{\rm int}}

\newcommand{\sd}{{\rm sd}}

\newcommand{\supp}{{\rm supp}}

\newcommand{\dD}{{\mathcal D}}
\newcommand{\eE}{{\mathcal E}}
\newcommand{\fF}{{\mathcal F}}
\newcommand{\hH}{{\mathcal H}}
\newcommand{\iI}{{\mathcal I}}
\newcommand{\jJ}{{\mathcal J}}
\newcommand{\pP}{{\mathcal P}}
\newcommand{\lL}{{\mathcal L}}
\newcommand{\qQ}{{\mathcal Q}}

\newcommand{\RR}{{\mathbb R}}
\newcommand{\fS}{{\mathfrak S}}
\newcommand{\NN}{{\mathbb N}}
\newcommand{\ZZ}{{\mathbb Z}}

\renewcommand{\to}{\rightarrow}

\newcommand{\sm}{{\smallsetminus}}
\begin{document}
\title[Face numbers of uniform triangulations]
{Face numbers of uniform triangulations of 
simplicial complexes}

\author{Christos~A.~Athanasiadis}

\address{Department of Mathematics\\
National and Kapodistrian University of Athens\\
Panepistimioupolis\\
15784 Athens, Greece}
\email{caath@math.uoa.gr}


\begin{abstract}
A triangulation of a simplicial complex $\Delta$ is 
called uniform if the $f$-vector of its restriction 
to a face of $\Delta$ depends only on the dimension of 
that face. This paper proves that the entries of the 
$h$-vector of a uniform triangulation of $\Delta$ can 
be expressed as nonnegative integer linear combinations
of those of the $h$-vector of $\Delta$, where the
coefficients depend only on the dimension of $\Delta$
and the $f$-vectors of the restrictions of the 
triangulation to simplices of various dimensions. 
Moreover, it provides information about these 
coefficients, including 
formulas, recurrence relations and various 
interpretations, and gives a criterion for the 
$h$-polynomial of a uniform triangulation to be 
real-rooted. These results unify and generalize several 
results in the literature about special types of 
triangulations, such as barycentric, edgewise and 
interval subdivisions.
\end{abstract}

\maketitle

\section{Introduction}
\label{sec:intro}

The study of triangulations of simplicial complexes 
from a face enumeration point of view was pioneered by 
Stanley~\cite{Sta92} \cite[Section~III.10]{StaCCA}. The 
main objective of~\cite{Sta92} was to understand the 
effect that various types of subdivision, including 
triangulations, have on the $h$-vector (a certain 
linear transformation of the face vector) of a 
simplicial complex; see \cite[Chapter~II]{StaCCA} 
for the importance of $h$-vectors on the face 
enumeration of simplicial complexes.

The transformation of the $h$-vector has been studied 
for specific triangulations since then, beginning with 
the work of Brenti--Welker~\cite{BW08} on barycentric 
subdivisions. These authors showed that the $h$-vector 
of the barycentric subdivision of a simplicial complex 
$\Delta$ is given by a nonnegative integer linear 
transformation of the $h$-vector of $\Delta$, which 
depends only on the dimension of $\Delta$. They also 
provided a combinatorial interpretation of the 
coefficients in terms of permutation enumeration. 
Analogous results have been proven for edgewise 
subdivisions~\cite{BW09}, partial barycentric 
subdivisions~\cite{AW18}, interval 
subdivisions~\cite{AN20} and antiprism 
triangulations~\cite{ABK20+}. The main result 
of~\cite{BW08} states that the $h$-polynomial (the 
generating polynomial for the $h$-vector) of the 
barycentric subdivision of $\Delta$ has only real 
roots (in particular, log-concave and unimodal 
coefficients) for every simplicial complex $\Delta$ 
with nonnegative $h$-vector. An analogous statement 
for edgewise subdivision follows from a result of 
Jochemko~\cite{Jo18} on the Veronese construction 
for rational formal power series, which improved 
earlier results by Brenti--Welker~\cite{BW09} and 
Beck--Stapledon~\cite{BS10}. 

The present paper aims to provide a common framework 
to explain and generalize these results. Given $d \in 
\NN \cup \{ \infty\}$, a triangular array $\fF$ of 
numbers $f_\fF(i,j)$ for $0 \le 
i \le j \le d$ and a simplicial complex $\Delta$ of 
dimension less than $d$, we will say that a triangulation 
$\Delta'$ of $\Delta$ is \emph{$\fF$-uniform} if for all 
$i, j$, the restriction of $\Delta'$ to any 
$(j-1)$-dimensional face of $\Delta$ has exactly 
$f_\fF(i,j)$ faces of dimension $i-1$. We will refer to 
the array $\fF$ as an \emph{$f$-triangle} of size $d$ 
and will call $\Delta'$ \emph{uniform}, if it is 
$\fF$-uniform for some $\fF$. All aforementioned examples 
of triangulations studied in the literature are uniform, 
since they have the stronger property that their 
restrictions to faces of $\Delta$ of the same dimension 
are combinatorially isomorphic. The following statement 
is the first main contribution of this paper. We use the
convention that $\{0, 1,\dots,d\} := \NN = \{0, 1, 
2,\dots\}$, when $d = \infty$.
\begin{theorem} \label{thm:mainA} 
Let $\fF$ be an $f$-triangle of size $d$. There exist 
nonnegative integers $p_\fF(n,k,j)$ for $n \in \{0, 
1,\dots,d\}$ and $k, j \in \{0, 1,\dots,n\}$, such that 
for all $n \le d$,
\begin{equation}
\label{eq:mainA}
h_j(\Delta') \ = \ \sum_{k=0}^n p_\fF(n,k,j) h_k(\Delta) 
\end{equation} 
for every $(n-1)$-dimensional simplicial complex $\Delta$,
every $\fF$-uniform triangulation $\Delta'$ of $\Delta$ 
and all $j \in \{0, 1,\dots,n\}$.
\end{theorem}

Given the main results of~\cite{BW08, Jo18} on 
real-rootedness, it seems natural to ask which uniform 
triangulations transform $h$-polynomials with nonnegative 
coefficients into polynomials with only real 
(necessarily negative) roots. Our second main 
contribution is a partial answer to 
this question which easily applies to barycentric and 
edgewise subdivisions (see Section~\ref{sec:apps}). Given 
an $f$-triangle $\fF$, we will denote by $h_\fF(\sigma_n, 
x)$ and $h_\fF(\partial \sigma_n, x)$ the $h$-polynomial of 
any $\fF$-uniform triangulation of the $(n-1)$-dimensional 
simplex $\sigma_n$ and its boundary complex, respectively
(by Theorem~\ref{thm:mainA}, these polynomials depend only 
on $n$ and $\fF$). 
\begin{theorem} \label{thm:mainB} 
Let $\fF$ be an $f$-triangle of size $d \in \NN$ and 
assume the following: 

\begin{itemize}
\itemsep=0pt
\item[{\rm (i)}]
$h_\fF(\sigma_n, x)$ is a real-rooted polynomial for all 
$n < d$.

\item[{\rm (ii)}]
$h_\fF(\sigma_n, x) - h_\fF(\partial \sigma_n, x)$ is 
either identically zero, or a real-rooted polynomial 
of degree $n-1$ with nonnegative coefficients which is 
interlaced by $h_\fF(\sigma_{n-1}, x)$, for all $n \le d$.
\end{itemize}
Then, for every $(d-1)$-dimensional simplicial complex 
$\Delta$ with nonnegative $h$-vector, the polynomial 
$h(\Delta', x)$ is real-rooted for every $\fF$-uniform 
triangulation $\Delta'$ of $\Delta$.
\end{theorem}

We now provide some more details about the content, 
methods and structure of this paper. 
Section~\ref{sec:pre} includes preliminaries on 
simplicial complexes, triangulations and their face 
enumeration. Section~\ref{sec:uniform} discusses 
uniform triangulations, their basic properties 
and motivating examples, given by barycentric and 
edgewise subdivisions and their variations and 
generalizations. A similar concept was introduced 
in \cite[Section~5]{DPS12} in order to study the 
asymptotics of the roots of the $h$-polynomial of a 
simplicial complex after iterated simplicial 
subdivision.

Theorem~\ref{thm:mainA} is proven in 
Section~\ref{sec:transform} and some immediate 
consequences are drawn (see Corollary~\ref{cor:mainA}). 
An explicit formula 
(Equation~(\ref{eq:pformula})) and further information
(see Proposition~\ref{prop:coef}), including a symmetry 
property and a universal recurrence relation, are given 
there for the coefficients which appear in the 
transformation 
(\ref{eq:mainA}) of the $h$-vector of a simplicial 
complex $\Delta$ under uniform triangulation. Perhaps 
not surprisingly, their nonnegativity follows from that 
of the local $h$-polynomials \cite[Section~4]{Sta92} of 
the restrictions of the triangulation to the faces of 
$\Delta$. In particular, the results of this paper are 
valid more generally for the class of quasi-geometric 
simplicial subdivisions \cite{Sta92}, which includes 
that of geometric simplicial subdivisions 
(triangulations), discussed here. One should note that
for the special types of triangulations treated in 
\cite{AW18, AN20, BW08, BW09}, the nonnegativity of the 
coefficients in (\ref{eq:mainA}) is proven there by 
finding 
explicit combinatorial interpretations by various 
techniques (a task which has required considerable 
effort in each case).
This is achieved in Section~\ref{sec:transform} for the 
$r$-colored barycentric subdivision by exploiting the 
universal recurrence (see 
Proposition~\ref{prop:csdr-coef}). 

Section~\ref{sec:operator} extends the subdivision 
operator on polynomials, associated with barycentric 
subdivisions (see \cite[Section~4]{Bra06} 
\cite[Section~7.3.3]{Bra15} and references therein),
to arbitrary uniform triangulations. The main properties
of this operator, developed there, yield two
interpretations of the coefficients in transformation 
(\ref{eq:mainA}) (see Propositions~\ref{prop:hsdop} 
and~\ref{prop:relative}). One of them is exploited in 
order to give a short proof of the recurrence relation
for these coefficients (see 
Corollary~\ref{cor:p-char}). The other expresses them 
as entries of $h$-vectors of Cohen--Macaulay relative 
simplicial complexes and results in a new proof of 
their nonnegativity (see 
Remark~\ref{rem:nonnegativity}).

Section~\ref{sec:roots} proves Theorem~\ref{thm:mainA} 
using the theory of interlacing polynomials. 
Section~\ref{sec:apps} includes applications. Among  
others, it recovers the main result of 
Brenti--Welker~\cite{BW08} on barycentric subdivisions 
and (partially) that of Jochemko~\cite{Jo18} on edgewise 
subdivisions and the Veronese construction for formal 
power series (see Examples~\ref{ex:sd} 
and~\ref{ex:esdr}). For an approach via shellability 
which applies to these two situations, as well as to 
barycentric subdivisions of certain cubical polytopes, 
see~\cite{HS20+}. Section~\ref{sec:apps} 
also deduces that certain polynomials which have 
appeared in the combinatorial literature have 
nonnegative, real-rooted and interlacing symmetric 
decompositions. Section~\ref{sec:rem} concludes with 
some comments and directions for further research.

\section{Preliminaries}
\label{sec:pre}

This section recalls definitions and background on
the face enumeration of (abstract) simplicial complexes 
and their triangulations. Any undefined terminology can
be found in~\cite{Sta92, StaCCA}. All simplicial 
complexes we consider will be finite. Throughout this 
paper, we denote by $\sigma_n$ the (simplicial complex 
of faces of the) standard $(n-1)$-dimensional simplex 
in $\RR^n$. We also set $\NN = \{0, 1, 2,\dots\}$, we 
recall our convention from the introduction that $\{0, 
1,\dots,\infty\} = \NN$ and 
for $d \in \NN \cup \{ \infty\}$, we denote by $\RR_d
[x]$ the real vector space of polynomials of degree at 
most $d$.

The \emph{$f$-polynomial} of an $(n-1)$-dimensional 
(abstract) simplicial complex $\Delta$ is defined as
$f(\Delta, x) := \sum_{i=0}^n f_{i-1} (\Delta) x^i$,
where $f_i (\Delta)$ is the number of 
$i$-dimensional faces of $\Delta$. The numerical 
sequence $f(\Delta) := (f_{-1}(\Delta), 
f_0(\Delta),\dots,f_{n-1}(\Delta))$ is called the 
\emph{$f$-vector}. The \emph{$h$-polynomial} of 
$\Delta$ is then defined by the formula
\begin{eqnarray*}
h(\Delta, x) & := & (1-x)^n f(\Delta, \frac{x}{1-x}) 
\ = \ \sum_{i=0}^n f_{i-1} (\Delta) \, x^i 
(1-x)^{n-i} \\ & := & \sum_{i=0}^n \, h_i (\Delta) 
x^i 
\end{eqnarray*}
and the sequence $h(\Delta) := (h_0(\Delta), h_1
(\Delta),\dots,h_n(\Delta))$ is called the 
\emph{$h$-vector} of $\Delta$. The polynomial 
$h(\Delta, x)$ has nonnegative coefficients for 
every $\Delta$ which is Cohen--Macaulay over some 
field. This happens, in particular, if $\Delta$ 
triangulates a ball or sphere (meaning that its
geometric realization is homeomorphic to a ball 
or sphere). Moreover, $h(\Delta, x)$ is symmetric, 
with center of symmetry $n/2$ (meaning that $h_i
(\Delta) = h_{n-i}(\Delta)$ for $0 \le i \le n$),
if $\Delta$ triangulates a sphere. For more 
information on these topics, we refer the reader 
to \cite[Chapter~II]{StaCCA}.

A \emph{relative simplicial complex} 
\cite[Section~III.7]{StaCCA} is any pair $(\Delta, 
\Gamma)$, where $\Delta$ is a simplicial complex 
and $\Gamma$ is a subcomplex of $\Delta$. The 
definitions of the $f$- and $h$-polynomial extend
naturally to this setting. Following 
\cite[Section~2]{BS20}, we say that $h(x) := 
(1-x)^n f(x/(1-x))$ is the \emph{$h$-polynomial}
associated to $f(x) \in \RR_n[x]$ (with respect to
$n$). Equivalently, $f(x) = (1+x)^n h(x/(1+x))$ is 
the \emph{$f$-polynomial} associated to $h(x) \in 
\RR_n[x]$ (with respect to $n$). The $h$-polynomial 
$h(\Delta / \Gamma, x)$ of the relative complex 
$(\Delta, \Gamma)$ is then defined as the $h$-polynomial
associated to the $f$-polynomial $f(\Delta / \Gamma, 
x) := \sum_{i=0}^n f_{i-1} (\Delta / \Gamma) x^i$,
where $n-1$ is the dimension of $\Delta$ and $f_i 
(\Delta / \Gamma)$ is the number of $i$-dimensional 
faces of $\Delta$ which do not belong to $\Gamma$. 
In particular, this defines the $h$-polynomial of
the interior $\inte(\Delta) = \Delta \sm \partial
\Delta$ of $\Delta$, when $\Delta$ triangulates a 
ball and $\partial \Delta$ is its boundary complex.

The following statement is a special case of 
\cite[Lemma~6.2]{Sta87}.
\begin{proposition} \label{prop:hsymmetry} 
{\rm (\cite{Sta87})}
Let $\Delta$ be a triangulation of an 
$(n-1)$-dimensional ball. Let $\Gamma$ be a 
subcomplex of $\partial\Delta$ which is homeomorphic 
to an $(n-2)$-dimensional ball or sphere and 
$\bar{\Gamma}$ be the subcomplex of $\partial\Delta$ 
whose facets are those of $\partial\Delta$ which do 
not belong to $\Gamma$. Then,
\[ x^n h(\Delta / \Gamma, 1/x) \ = \ h(\Delta /
   \bar{\Gamma}, x). \]
In particular, $x^n h(\Delta, 1/x) = h(\inte(\Delta), 
x)$.
\end{proposition}

Consider two simplicial complexes 
$\Delta$ and $\Delta'$. We say that $\Delta'$ is a 
\emph{triangulation} of $\Delta$ if there exist 
geometric realizations $K'$ and $K$ of $\Delta'$ and 
$\Delta$, respectively, such that $K'$ geometrically 
subdivides $K$. Given a simplex $L \in K$ with 
corresponding face $F \in \Delta$, the triangulation 
$K'$ naturally restricts to a triangulation $K'_L$ of 
$L$. The subcomplex $\Delta'_F$ of $\Delta'$ 
corresponding to $K'_L$ is a triangulation of the 
abstract simplex $2^F$, called the \emph{restriction} 
of $\Delta'$ to $F$. The \emph{carrier} of a face 
$G \in \Delta'$ is defined as the smallest face $F 
\in \Delta$ such that $G \in \Delta'_F$. 

Associated to the restrictions $\Delta'_F$ are 
certain enumerative invariants, called local 
$h$-polynomials. Given any triangulation $\Gamma$ of 
an $(n-1)$-dimensional simplex with vertex set $V$, 
the \emph{local $h$-polynomial} of $\Gamma$ (with
respect to $V$) is defined 
\cite[Definition~2.1]{Sta92} by the formula 
\[ \ell_V (\Gamma, x) \ = \sum_{F \subseteq V} 
  \, (-1)^{n - |F|} \, h (\Gamma_F, x). \]
By the principle of inclusion-exclusion, we have 
\[ h (\Gamma, x) \ = \sum_{F \subseteq V} 
  \ell_F (\Gamma_F, x). \]
The polynomial $\ell_V (\Gamma, x)$ is 
symmetric, with center of symmetry $n/2$, and has
nonnegative coefficients; it plays a fundamental 
role in the enumerative theory of simplicial 
subdivisions \cite{Sta92} 
\cite[Section~III.10]{StaCCA} and in the proof of 
Theorem~\ref{thm:mainA} as well.

\section{Uniform triangulations}
\label{sec:uniform}

This section introduces uniform triangulations, fixes 
related notation and terminology and discusses basic 
properties and the main examples of interest.

Let us fix a number $d \in \NN \cup \{ \infty \}$ once
and for all; for the main applications, one can always 
take $d = \infty$. An \emph{$f$-triangle} of size $d$
will be a triangular array $\fF = (f_\fF(i,j))_{0 \le i 
\le j \le d}$ of nonnegative integers (where $i, j$ are
finite numbers). We will say that
\begin{equation}
\label{eq:f(sigma)}
f_\fF(\sigma_n, x) \ := \ \sum_{i=0}^n f_\fF(i,n) x^i
\end{equation} 
is the $n$th \emph{$f$-polynomial} associated to $\fF$.
\begin{definition} \label{def:uniform} 
Let $\fF$ be an $f$-triangle of size $d$ and $\Delta$ 
be a simplicial complex of dimension less than $d$. A 
triangulation $\Delta'$ of $\Delta$ is called 
\emph{$\fF$-uniform} if $f(\Delta'_F, x) = f_\fF 
(\sigma_n,x)$ for every $(n-1)$-dimensional face 
$F \in \Delta$ and all $n \le d$. 
\end{definition}

Equivalently, we require that for all $0 \le i \le j 
\le d$, the restriction of $\Delta'$ to any face of 
$\Delta$ of dimension $j-1$ has exactly $f_\fF(i,j)$ 
faces of dimension $i-1$. This paper is not concerned
with the problem to determine necessary and sufficient
conditions on $\fF$, so that $\fF$-uniform 
triangulations of $(n-1)$-dimensional simplices exist
for all $n \le d$. We will then say that $\fF$ is 
\emph{feasible}. Most statements in this
paper are vacuously true for non-feasible triangles 
$\fF$. As the notation in Equation~(\ref{eq:f(sigma)}) 
suggests, when we talk about the face enumeration of 
an $\fF$-uniform triangulation of $\sigma_n$, we will 
mean that of the restriction of $\Delta'$ to any 
$(n-1)$-dimensional face of $\Delta$. 

Clearly, if a 
triangulation $\Delta'$ of $\Delta$ is $\fF$-uniform, 
then so is the restriction of $\Delta'$ to any 
subcomplex of $\Delta$. 

\subsection{Face triangles}
\label{sec:triangles}

A feasible $f$-triangle $\fF$ of size $d$ gives rise 
to the \emph{$h$-polynomials, interior $f$-polynomials, 
interior $h$-polynomials} and \emph{local $h$-polynomials}
\begin{eqnarray*}
h_\fF(\sigma_n, x) & = & \sum_{i=0}^n h_\fF(i,n) x^i 
\label{eq:h(sigma)}, \\
f^\circ_\fF(\sigma_n, x) & = & \sum_{i=0}^n 
f^\circ_\fF(i,n) x^i \label{eq:fcirc(sigma)}, \\
h^\circ_\fF(\sigma_n, x) & = & \sum_{i=0}^n 
h^\circ_\fF(i,n) x^i \label{eq:hcirc(sigma)}, \\
\ell_\fF(\sigma_n, x) & = & \sum_{i=0}^n \ell_\fF(i,n) 
x^i \label{eq:ell(sigma)} 
\end{eqnarray*} 
for $0 \le n \le d$, respectively. These are 
defined as the $h$-polynomial, interior $f$-polynomial, 
interior $h$-polynomial and local $h$-polynomial of 
any $\fF$-uniform triangulation of the simplex $\sigma_n$. 
They can be represented by the corresponding
$h$-triangle $\hH = (h_\fF(i,j))$, interior $f$-triangle 
$\fF^\circ = (f^\circ_\fF(i,j))$, interior $h$-triangle 
$\hH^\circ = (h^\circ_\fF(i,j))$ and local $h$-triangle 
$\lL = (\ell_\fF(i,j))$ of size $d$, respectively. For
example, $f^\circ_\fF(i,n)$ is equal to the number of 
interior $(i-1)$-dimensional faces of any $\fF$-uniform 
triangulation of $\sigma_n$ (or of any other simplex of
the same dimension). 
These triangles determine each other, since the 
corresponding polynomials are related by the invertible
transformations

\begin{eqnarray}
h_\fF(\sigma_n, x) & = & (1-x)^n 
f_\fF(\sigma_n, \frac{x}{1-x}), \label{eq:hf(sigma)} \\
h^\circ_\fF(\sigma_n, x) & = & 
x^n h_\fF(\sigma_n, \frac{1}{x}), \label{eq:hhcirc(sigma)} \\
f^\circ_\fF(\sigma_n, x) & = & (1+x)^n h^\circ_\fF
(\sigma_n, \frac{x}{1+x}), \label{eq:fcirchcirc(sigma)}\\
\ell_\fF(\sigma_n, x) & = & \sum_{k=0}^n (-1)^{n-k} 
{n \choose k} h_\fF(\sigma_k, x). \label{eq:ellh(sigma)}
\end{eqnarray} 

\medskip
These equalities define the triangles $\hH$, $\fF^\circ$, 
$\hH^\circ$ and $\lL$ even when $\fF$ may not be feasible.
The second follows from the last sentence of 
Proposition~\ref{prop:hsymmetry} and yields the relation
\begin{equation}
\label{eq:ffcirc}
f^\circ_\fF(\sigma_n, -1-x) \ = \ (-1)^n 
f_\fF(\sigma_n, x)
\end{equation} 
between the $f$-polynomials and interior $f$-polynomials 
associated to $\fF$.
\begin{example} \label{ex:d=2} \rm
Let $d=2$ and suppose that $\fF$ comes from the uniform
triangulation which subdivides any 1-simplex into $r$
such simplices, by inserting $r-1$ interior vertices. 
Then, the various triangles we have defined are given by:

\medskip
\begin{itemize}
\itemsep=0pt
\item[$\bullet$]
$f_\fF(\sigma_0, x) = 1, f_\fF(\sigma_1, x) = 1+x, 
f_\fF(\sigma_2, x) = 1 + (r+1)x + rx^2$,

\item[$\bullet$]
$f^\circ_\fF(\sigma_0, x) = 1, f^\circ_\fF(\sigma_1, x) 
= x, f^\circ_\fF(\sigma_2, x) = (r-1)x + rx^2$,

\item[$\bullet$]
$h_\fF(\sigma_0, x) = h_\fF(\sigma_1, x) = 1, 
h_\fF(\sigma_2, x) = 1 + (r-1)x$,

\item[$\bullet$]
$h^\circ_\fF(\sigma_0, x) = 1, h^\circ_\fF(\sigma_1, x) 
= x, h^\circ_\fF(\sigma_2, x) = (r-1)x + x^2$, and

\item[$\bullet$]
$\ell_\fF(\sigma_0, x) = 1, \ell_\fF(\sigma_1, x) = 0,
\ell_\fF(\sigma_2, x) = (r-1)x$.
\end{itemize}
\end{example}

\subsection{Examples}
\label{sec:examples}

We now briefly review important examples of 
uniform triangulations which have already been studied 
in the literature, in terms of their face enumeration,
and have provided much of the motivation behind this 
paper. The corresponding $f$-triangles can be considered 
to be of infinite size.

At the core of our examples lie barycentric and 
edgewise subdivisions. Let $\Delta$ be a simplicial 
complex of dimension $n-1$ with vertex set $V(\Delta)$ 
and $r$ be a positive integer. The \emph{barycentric 
subdivision} of $\Delta$ is denoted by $\sd(\Delta)$ 
and defined as the simplicial complex of all chains in 
the poset of nonempty faces of $\Delta$. The edgewise 
subdivision depends on $r$ and a linear ordering of 
$V(\Delta)$ (although its face vector is independent 
of the latter). Given such an ordering 
$v_1, v_2,\dots,v_m$, denote by $V_r(\Delta)$ the set 
of maps $f: V(\Delta) \to \NN$ such that $\supp(f) \in 
\Delta$ and $f(v_1) + f(v_2) + \cdots + f(v_m) = r$, 
where $\supp(f)$ is the set of all $v \in V(\Delta)$ 
for which $f(v) \ne 0$. For $f \in V_r(\Delta)$, let 
$\iota(f): V(\Delta) \to \NN$ be the map defined by 
setting $\iota(f)(v_j) = f(v_1) + f(v_2) + 
\cdots + f(v_j)$ for $j \in \{1, 2,\dots,m\}$. The 
\emph{$r$-fold edgewise subdivision} of $\Delta$, 
denoted by $\esd_r(\Delta)$, is the simplicial complex  
on the vertex set $V_r(\Delta)$ of which a set 
$E \subseteq V_r(\Delta)$ is a face if the following 
two conditions are satisfied:

\begin{itemize}
\itemsep=0pt
\item[$\bullet$]
$\bigcup_{f \in E} \, \supp(f) \in \Delta$ and

\item[$\bullet$]
$\iota(f) - \iota(g) \in \{0, 1\}^{V(\Delta)}$, or $\iota(g) 
- \iota(f) \in \{0, 1\}^{V(\Delta)}$, for all $f, g \in E$.
\end{itemize}

The simplicial complexes $\sd(\Delta)$ and $\esd_r(\Delta)$ 
can be realized as triangulations of $\Delta$. This is 
elementary and well known for the former, but slightly less 
obvious for the latter; see \cite[Section~5]{Ath20} and 
references therein. The 4-fold edgewise subdivision of the 
2-simplex is shown on Figure~1; it will be used as a 
running example in this paper.  

\begin{figure}
\begin{center}
\begin{tikzpicture}[scale=0.85]
\label{fg:4fold}

  \draw(0,0) -- (4,0) -- (4,4) -- (0,0);
  \draw(1,1) -- (4,1) -- (3,0) -- (3,3) 
             -- (4,3) -- (1,0) -- (1,1); 
  \draw(2,2) -- (4,2) -- (2,0) -- (2,2); 

\end{tikzpicture}
\caption{The 4-fold edgewise subdivision of the 
2-simplex}
\end{center}
\end{figure}
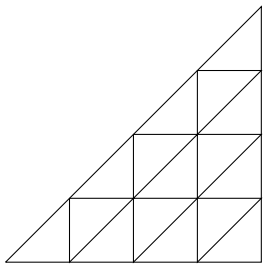

Barycentric and edgewise subdivisions can be 
combined to form the \emph{$r$-colored barycentric 
subdivision} of $\Delta$. This triangulation is
defined as the $r$-fold edgewise subdivision of 
$\sd(\Delta)$; it was introduced in \cite{Ath14} 
in order to partially interpret geometrically 
the derangement polynomial for the colored 
permutation group $\ZZ_r \wr \fS_n$. The 
enumerative combinatorics of the $r$-colored 
barycentric subdivision relates to that of 
$r$-colored permutations, just as the enumerative 
combinatorics of barycentric subdivision relates 
to that of usual (uncolored) permutations; see 
\cite{Ath14} \cite[Section~5]{Ath20}. Clearly, it
reduces to $\sd(\Delta)$ for $r=1$. As explained
in \cite[Remark~4.5]{Ath16a}, for $r=2$ it has 
the same $f$-vector as another interesting
triangulation of $\Delta$, namely the 
\emph{interval triangulation}. As a simplicial
complex, the latter consists of all chains of 
nonempty closed intervals in the poset of nonempty 
faces of $\Delta$; it can also be described as a 
cubical (or signed) analogue of barycentric 
subdivision. The face enumeration of the interval 
triangulation was studied in~\cite{Sav13}, when 
$\Delta$ is a simplex (see also 
\cite[Section~4]{Ath16a} \cite[Section~3.3]{Ath18} 
\cite[Section~5]{Ath20a}), and in \cite{AN20} 
for arbitrary $\Delta$.

The transformation of the $h$-vector of $\Delta$ 
under barycentric, edgewise and interval subdivision
was studied in \cite{BW08, BW09, AN20}, respectively,
where combinatorial interpretations of the 
coefficients which appear in 
Equation~(\ref{eq:mainA}) were found. By a simple 
application of the recurrence (see 
Proposition~\ref{prop:coef}) for these coefficients,
we generalize the interpretations in the first and 
third cases to that of the $r$-colored barycentric 
subdivision in Section~\ref{sec:transform} (see 
Proposition~\ref{prop:csdr-coef}). Useful
formulas for the $h$-polynomial of the barycentric 
and edgewise subdivisions of an $(n-1)$-dimensional 
simplicial complex $\Delta$ are
\begin{equation}
\label{eq:BW08}
\sum_{m \ge 0} \left( \, \sum_{i=0}^n h_i(\Delta) 
m^i(m+1)^{n-i} \right) x^m \ = \ 
\frac{h(\sd(\Delta), x)}{(1-x)^{n+1}}
\end{equation} 
(see \cite[Equation~(3.5)]{BW08}) and
\begin{equation}
\label{eq:BW09}
h(\esd_r(\Delta), x) \ = \ \left(
   (1 + x + x^2 + \cdots + x^{r-1})^n h(\Delta, x) 
	 \right)^{\langle r,0 \rangle} 
\end{equation}
(see \cite[Section~4]{Ath14} and references therein), 
where we have used the standard notation
\[ g(x) \ = \ g^{\langle r,0 \rangle}(x^r) + x 
g^{\langle r,1 \rangle}(x^r) + \cdots	+ x^{r-1} 
g^{\langle r,r-1 \rangle}(x^r) \]
for $g(x) \in \RR[x]$. The main results 
of~\cite{BW08} and~\cite{Jo18}, respectively, show 
that barycentric subdivision transforms $h$-polynomials 
with nonnegative coefficients to polynomials with only 
real roots and that the $r$-fold edgewise subdivision 
has this property, provided $r$ is larger than the 
dimension of $\Delta$. We deduce these results from 
Theorem~\ref{thm:mainB} in Section~\ref{sec:apps}.
The fact that the $r$-colored barycentric subdivision
has the same property for every $r$ follows from the
result of~\cite{BW08} for barycentric subdivisions
and \cite[Theorem~4.5.6]{Bre89} or 
\cite[Corollary~3.4]{Zh20} (see 
Proposition~\ref{prop:csdr}).

All triangulations of $\Delta$ we have discussed 
here are uniform, since their restrictions to 
faces of $\Delta$ are triangulations of the same 
kind. For more information on their combinatorics (in
particular, for combinatorial interpretations of 
their local $h$-polynomials),
see \cite[Section~4]{Ath16a} \cite[Section~3.3]{Ath18} 
and references therein. Another very interesting 
and motivating example of uniform triangulation, 
namely the antiprism triangulation
\cite[Section~Appendix~A]{IJ03}, has been studied 
more recently in~\cite{ABK20+}.

\section{Transformation of the $h$-vector}
\label{sec:transform}

This section proves the following more detailed 
version of Theorem~\ref{thm:mainA}, which describes 
the transformation of the $h$-vector of a simplicial 
complex under uniform triangulation, and studies
the coefficients which appear in 
Equation~(\ref{eq:mainA}). A combinatorial 
interpretation of these coefficients is deduced from 
the main recurrence they satisfy, in the special case 
of the $r$-colored barycentric subdivision. 
\begin{theorem} \label{thm:mainAB} 
Let $\fF$ be an $f$-triangle of size $d \in \NN \cup 
\{ \infty\}$. For $n, k \in \{0, 1,\dots,d\}$ with 
$k \le n$, there exist polynomials 
\begin{equation}
\label{eq:pdef}
p_{\fF,n,k}(x) \ = \ \sum_{j=0}^n p_\fF (n,k,j) x^j 
\end{equation} 
with nonnegative integer 
coefficients, such that the following holds for all 
$n \le d$: the $h$-polynomial of any $\fF$-uniform 
triangulation of any $(n-1)$-dimensional simplicial 
complex $\Delta$ is equal to
\begin{equation}
\label{eq:hFdef}
h_\fF(\Delta, x) \ := \ \sum_{k=0}^n h_k(\Delta)
p_{\fF,n,k} (x).
\end{equation} 
The explicit formula 
\begin{equation}
\label{eq:pformula}
p_{\fF,n,k}(x) \ = \ \sum_{r=0}^n \ell_\fF(\sigma_r,x) 
\sum_{i=0}^r {n-k \choose i}{k \choose r-i} x^{k-r+i}
\end{equation} 
holds for all $n, k$.
\end{theorem}

\begin{example} \label{ex:d=3} \rm
For $d \ge 2$ we have $\ell_\fF(\sigma_0, x) = 1$, 
$\ell_\fF(\sigma_1, x) = 0$ and $\ell_\fF(\sigma_2, x) 
= (r-1)x$, where $r-1$ is the number of interior 
vertices of any $\fF$-uniform triangulation of the
1-simplex. Equation~(\ref{eq:pformula}) yields that 
$p_{\fF,0,0} (x) = p_{\fF,1,0} (x) = 1$, $p_{\fF,1,1} 
(x) = x$ and 
\[ p_{\fF,2,k} (x) \ = \ \begin{cases}
    1 + (r-1)x, & \text{if $k=0$} \\
    rx, & \text{if $k=1$} \\
    (r-1)x + x^2, & \text{if $k=2$}. \end{cases} \]

Letting $d = 3$ and assuming that $\fF$ is the
$f$-triangle for the edgewise subdivision of Figure~1,
we also have $\ell_\fF(\sigma_2, x) = 3x$ and 
$\ell_\fF(\sigma_3, x) = 3x + 3x^2$ and compute that 
\[ p_{\fF,3,k} (x) \ = \ \begin{cases}
    1 + 12x + 3x^2, & \text{if $k=0$} \\
    10x + 6x^2, & \text{if $k=1$} \\
    6x + 10x^2, & \text{if $k=2$} \\
    3x + 12x^2 + x^3, & \text{if $k=3$}. \end{cases} \]
\end{example}

\medskip
Before proceeding with the proof of 
Theorem~\ref{thm:mainAB}, we establish the following 
combinatorial identity which will be used there. We 
adopt the standard convention about binomial coefficients 
that ${n \choose k} := 0$, if $k$ is not in the range 
$0 \le k \le n$.
\begin{lemma} \label{lem:identity} 
We have
\[ \sum_{m=0}^n {m \choose r}{n-k \choose m-k} 
x^{m-r} (1-x)^{n-m} \ = \ \sum_{i=0}^r 
{n-k \choose i}{k \choose r-i} x^{k-r+i} \]
for every $n \in \NN$ and all $r, k \in \{0, 1,\dots,n\}$.
\end{lemma}
\begin{proof}
Let us denote the left-hand side by $L(n,k,r)$. Shifting 
the index $m$ to $m+k$ and using the identity 
${m+k \choose r} = \sum_{i=0}^r {m \choose i}
{k \choose r-i}$, we get
\begin{eqnarray*}
L(n,k,r) & = & \sum_{m=0}^{n-k} {m+k \choose r}
{n-k \choose m} x^{m+k-r} (1-x)^{n-k-m} \\
& = & x^{k-r} \, \sum_{m=0}^{n-k} \left( \, 
\sum_{i=0}^r {m \choose i}{k \choose r-i} \right)
{n-k \choose m} x^m (1-x)^{n-k-m}
\\ & = & x^{k-r} \, \sum_{i=0}^r {k \choose r-i}
\sum_{m=0}^{n-k} {m \choose i}{n-k \choose m} x^m 
(1-x)^{n-k-m}.
\end{eqnarray*}
Using the identity ${m \choose i}{n-k \choose m} = 
{n-k \choose i}{n-k-i \choose m-i}$ and applying the 
binomial theorem shows that the inner sum equals 
${n-k \choose i} x^i$ and the proof follows.
\end{proof}

\medskip
\noindent
\emph{Proof of Theorem~\ref{thm:mainAB}.}
Let $\Delta'$ be an $\fF$-uniform triangulation of an 
$(n-1)$-dimensional simplicial complex $\Delta$.
Using the notation of Section~\ref{sec:triangles}, 
since $\Delta'$ is $\fF$-uniform, for every 
$(m-1)$-dimensional face $F \in \Delta$ there exist
exactly $f^\circ_\fF(j,m)$ faces of $\Delta'$ of 
dimension $j-1$ with carrier $F$. Therefore, 
\begin{equation}
\label{eq:fbasic}
f_{j-1}(\Delta') \ = \ \sum_{m=j}^n f_{m-1}(\Delta)
   \cdot f^\circ_\fF(j,m)
\end{equation} 
for every $j \in \{0, 1,\dots,n\}$ and hence
\begin{eqnarray*}
f(\Delta', x) & = & \sum_{j=0}^n f_{j-1}(\Delta') x^j 
\ = \ \sum_{j=0}^n \left( \, \sum_{m=j}^n f_{m-1}
(\Delta) \cdot f^\circ_\fF(j,m) \right) x^j \\
& = & \sum_{m=0}^n f_{m-1}(\Delta) 
\left( \, \sum_{j=0}^m f^\circ_\fF(j,m) x^j \right) 
\\ & = & \sum_{m=0}^n f_{m-1}(\Delta) \cdot 
f^\circ_\fF(\sigma_m, x).
\end{eqnarray*}
Applying the transformations between $f$-polynomials 
and $h$-polynomials and 
Equation~(\ref{eq:hhcirc(sigma)}), we conclude that

\begin{eqnarray*}
h(\Delta', x) & = & (1-x)^n f(\Delta', \frac{x}{1-x}) 
\\ & = & (1-x)^n \, \sum_{m=0}^n f_{m-1}(\Delta) \cdot 
f^\circ_\fF(\sigma_m, \frac{x}{1-x}) \\ & = & 
\sum_{m=0}^n (1-x)^{n-m} f_{m-1}(\Delta) \cdot 
h^\circ_\fF(\sigma_m, x) \\ & = & \sum_{m=0}^n x^m 
(1-x)^{n-m} f_{m-1}(\Delta) \cdot h_\fF(\sigma_m, 1/x)
\\ & = & \ \sum_{m=0}^n x^m (1-x)^{n-m} h_\fF(\sigma_m, 
1/x) \cdot \sum_{k=0}^m h_k(\Delta) {n-k \choose m-k} 
\\ & = & \sum_{k=0}^n h_k(\Delta) p_{\fF,n,k} (x),
\end{eqnarray*}
where
\[ p_{\fF,n,k} (x) \ := \ \sum_{m=k}^n {n-k \choose m-k}
   x^m (1-x)^{n-m} \, h_\fF(\sigma_m, 1/x). \]

\medskip
\noindent
Expressing $h_\fF(\sigma_m, 1/x)$ in terms of local 
$h$-polynomials and using the symmetry of the latter, we 
conclude further that

\begin{eqnarray*}
p_{\fF,n,k} (x) & = & \sum_{m=k}^n {n-k \choose m-k}
x^m (1-x)^{n-m} \sum_{r=0}^m {m \choose r} 
\ell_\fF(\sigma_r, 1/x) \\ & = & 
\sum_{r=0}^n \ell_\fF(\sigma_r, x) 
\sum_{m=\max\{r,k\}}^n {m \choose r}{n-k \choose m-k} 
x^{m-r} (1-x)^{n-m} \\ & = & 
\sum_{r=0}^n \ell_\fF(\sigma_r, x) \sum_{i=0}^r 
{n-k \choose i}{k \choose r-i} x^{k-r+i},
\end{eqnarray*}

\smallskip
\noindent
where the last equality follows from 
Lemma~\ref{lem:identity}. This computation, together 
with the nonnegativity of the coefficients of the 
polynomials $\ell_\fF(\sigma_r, x)$, imply all claims 
in the statement of the theorem.
\qed

\begin{remark} \label{rem:ffF} \rm
We will write $f_\fF(\Delta, x) := (1+x)^n h_\fF(\Delta, 
x/(1+x))$ for the $f$-polynomial corresponding to 
$h_\fF(\Delta, x)$. Thus, $f_\fF(\Delta, x)$ is equal 
to the $f$-polynomial of any $\fF$-uniform 
triangulation of any $(n-1)$-dimensional simplicial 
complex $\Delta$ and the coefficient of $x^j$ in $f_\fF 
(\Delta, x)$ is equal to the right-hand side of 
Equation~(\ref{eq:fbasic}). 
\qed
\end{remark}
The following corollary generalizes analogous statements 
for barycentric, edgewise and interval subdivisions 
\cite[Section~2]{BW08} \cite[Section~1]{BW09} 
\cite[Section~3]{AN20} to uniform triangulations.
\begin{corollary} \label{cor:mainA} 
Let $\fF$ be an $f$-triangle of size $d$ and $\Delta$ 
be an $(n-1)$-dimensional simplicial complex, for some 
$n \le d$. 

\begin{itemize}
\itemsep=0pt
\item[{\rm (a)}]
If $h(\Delta,x)$ is symmetric with center of symmetry 
$n/2$, then so is $h(\Delta',x)$ for every $\fF$-uniform 
triangulation $\Delta'$ of $\Delta$.

\item[{\rm (b)}]
If $h(\Delta,x)$ has nonnegative coefficients, then the 
inequality $h(\Delta,x) \le h(\Delta',x)$ holds 
coefficientwise for every $\fF$-uniform triangulation 
$\Delta'$ of $\Delta$.
\end{itemize}
\end{corollary}
\begin{proof}
To prove part (a), suppose that $h_k(\Delta) = h_{n-k}
(\Delta)$ for every $k \in \{0, 1,\dots,n\}$. In view 
of Theorem~\ref{thm:mainAB}, we need to 
show that $x^n h_\fF(\Delta, 1/x) = h_\fF(\Delta, x)$. 
Our assumption and Equation~(\ref{eq:hFdef}) imply that
\begin{eqnarray*}
x^n h_\fF(\Delta, 1/x) & = & \sum_{k=0}^n h_k(\Delta)
x^n p_{\fF,n,k} (1/x) \ = \ \sum_{k=0}^n h_{n-k}(\Delta)
x^n p_{\fF,n,n-k} (1/x) \\ & = & \sum_{k=0}^n h_k(\Delta)
x^n p_{\fF,n,n-k} (1/x)
\end{eqnarray*}
and thus, it suffices to verify that $x^n p_{\fF,n,n-k} 
(1/x) = p_{\fF,n,k} (x)$. Indeed, using
Equation~(\ref{eq:pformula}) and the symmetry of the 
polynomials $\ell_\fF(\sigma_r, x)$, we get

\begin{eqnarray*}
x^n p_{\fF,n,n-k} (1/x) & = & x^n \,
\sum_{r=0}^n \ell_\fF(\sigma_r, 1/x) 
\sum_{i=0}^r {k \choose i}{n-k \choose r-i} (1/x)^{n-k-r+i} 
\\ & = & \sum_{r=0}^n x^r \ell_\fF(\sigma_r, 1/x) 
\sum_{i=0}^r {k \choose i}{n-k \choose r-i} x^{k-i} \\
& = & \sum_{r=0}^n \ell_\fF(\sigma_r, x) 
\sum_{i=0}^r {k \choose r-i}{n-k \choose i} x^{k-r+i} \ = \ 
p_{\fF,n,k} (x).
\end{eqnarray*}

\medskip
For part (b), suppose that $h_k(\Delta) \ge 0$ for all 
$k$. By Theorem~\ref{thm:mainAB}, it suffices to show 
that the coefficient of $x^k$ in $p_{\fF,n,k}(x)$ is 
positive for every such $k$. Indeed, since $\ell_\fF
(\sigma_0,x) = 1$, the summand of the right-hand side 
of Equation~(\ref{eq:pformula}) corresponding to $r=0$ 
is equal to $x^k$ and the proof follows. 
\end{proof}

The following statement lists the main properties of
the coefficients $p_\fF(n,k,j)$.
\begin{proposition} \label{prop:coef} 
Let $\fF$ be a feasible $f$-triangle of size $d$ and 
$n \in \{0, 1,\dots,d\}$. 

\begin{itemize}
\itemsep=0pt
\item[{\rm (a)}]
We have $p_\fF(n,k,j) = p_\fF(n,n-k,n-j)$ for all 
$k, j \in \{0, 1,\dots,n\}$.

\item[{\rm (b)}]
The recurrence 
\begin{equation}
\label{eq:recurrence}
p_\fF(n,k,j) \ = \ p_\fF(n,k-1,j) + p_\fF(n-1,k-1,j-1) - 
p_\fF(n-1,k-1,j)
\end{equation} 
holds for all $k, j \in \{0, 1,\dots,n\}$ with $k \ge 
1$. 

\item[{\rm (c)}]
We have $p_\fF(n,0,j) = h_\fF(j,n)$ for every $j \in 
\{0, 1,\dots,n\}$. Equivalently, 
\[ \sum_{j=0}^n p_\fF(n,0,j) x^j \ = \ 
h_\fF(\sigma_n,x). \]

\item[{\rm (d)}]
For every $k \in \{0, 1,\dots,n\}$,
\[ \sum_{j=0}^n p_\fF(n,k,j) \ = \ h_\fF(\sigma_n,1).
\] 
Equivalently, the sum on the left-hand side is 
independent of $k$ and equal to the number of facets 
of any $\fF$-uniform triangulation of the simplex 
$\sigma_n$. 

\item[{\rm (e)}]
For every $j \in \{0, 1,\dots,n\}$,
\[ \sum_{k=0}^n p_\fF(n,k,j) \ = \ (h_\fF)_j (\partial
\sigma_{n+1}). \]
\end{itemize}
\end{proposition}
\begin{proof}
Part (a) can be restated as $p_{\fF,n,k} (x) = x^n 
p_{\fF,n,n-k} (1/x)$, which has already been shown in 
the proof of Corollary~\ref{cor:mainA}. Part (b) can 
be verified by direct computation; see the comment 
following this proof. Since a more conceptual  
argument is given in Section~\ref{sec:operator} 
(see Corollary~\ref{cor:p-char}), we postpone the 
proof of (b) until then. 

Parts (c) and (e) are direct consequences of 
Equation~(\ref{eq:hFdef}), applied when $\Delta =
\sigma_n$ and $\Delta = \partial\sigma_{n+1}$, 
respectively: we have $h_0(\Delta) = 1$ and $h_k(\Delta) 
= 0$ for $1 \le k \le n$ in the former case, and $h_k
(\Delta) = 1$ for all $0 \le k \le n$ in the latter. 
For part (d), we may compute directly from the 
formula~(\ref{eq:pformula}) that 

\begin{eqnarray*}
\sum_{j=0}^n p_\fF(n,k,j) & = & p_{\fF,n,k}(1) \ = \ 
\sum_{r=0}^n \ell_\fF(\sigma_r, 1) \cdot
\sum_{i=0}^r {n-k \choose i}{k \choose r-i} 
\\ & = & \sum_{r=0}^n \ell_\fF(\sigma_r, 1) \cdot
\frac{k!(n-k)!}{r!(n-r)!} \cdot
\sum_{i=0}^r {r \choose i}{n-r \choose n-k-i} \\
& = & \sum_{r=0}^n \ell_\fF(\sigma_r, 1) \cdot
\frac{k!(n-k)!}{r!(n-r)!} \cdot {n \choose n-k} \\ 
& = & \sum_{r=0}^n {n \choose r} \ell_\fF(\sigma_r, 1) 
\ = \ h_\fF(\sigma_n,1).
\end{eqnarray*}

\medskip
Alternatively, one can sum the recurrence of part (b)
with respect to $j$ and use induction on $k$, together
with the result of part (c).
\end{proof}

The recurrence of part (b) was discovered by J-M.~Brunink 
and M.~Juhnke-Kubitzke in the special case of antiprism 
triangulations (see \cite{ABK20+}) 
and proven by direct computation, based on an explicit 
expression for the coefficients $p_\fF(n,k,j)$ in terms 
of the $\fF$-triangle. 

We now derive a combinatorial interpretation of 
$p_\fF(n,k,j)$ from this recurrence in the special case 
of the $r$-colored barycentric subdivision. Recall that 
an $r$-colored permutation $w \in \ZZ_r \wr \fS_n$ is 
defined as a pair 
$(\tau, \varepsilon)$, where $\tau = (\tau(1), 
\tau(2),\dots,\tau(n))$ is a permutation of 
$\{1, 2,\dots,n\}$, $\varepsilon = (\varepsilon_1, 
\varepsilon_2,\dots,\varepsilon_n) \in \{0, 
1,\dots,r-1\}^n$ and $\varepsilon_i$ is thought of as 
the color assigned to $\tau(i)$. An index $1 \le i \le 
n$ is a \emph{descent} of $w$ if either $\varepsilon_i > 
\varepsilon_{i+1}$, or $\varepsilon_i = \varepsilon_{i+1}$ 
and $\tau(i) > \tau(i+1)$, where $\tau(n+1) := n+1$ and 
$\varepsilon_{n+1} := 0$ (in particular, $n$ is a descent 
of $w$ if and only if $\tau(n)$ has nonzero color). Our 
interpretation reduces to those of 
\cite[Theorem~2.2]{BW08} and \cite[Theorem~3.1]{AN20} 
for $r=1$ and $r=2$, respectively (which are obtained in
\cite{AN20, BW08} by different methods). 
\begin{proposition} \label{prop:csdr-coef} 
Let $\fF$ be the $f$-triangle for the $r$-colored 
barycentric subdivision. Then, $p_\fF(n,k,j)$ is equal
to the number of $r$-colored permutations $w \in \ZZ_r 
\wr \fS_{n+1}$ which have first coordinate of zero color,
$j$ descents and last coordinate of zero color and equal
to $n+1-k$.
\end{proposition}
\begin{proof}
Let $q(n,k,j)$ be the number of $r$-colored permutations,
described in the proposition. By Proposition~\ref{prop:coef} 
(c), $p_\fF(n,0,j)$ is equal to the coefficient of $x^j$ 
in $h_\fF(\sigma_n,x)$. This is known 
\cite[Proposition~5.1~(c)]{Ath20} to equal the number of 
colored permutations $w \in \ZZ_r \wr \fS_n$ which have 
first coordinate of zero color and $j$ descents. This 
number is equal to $q(n,0,j)$. Thus, it
suffices to show that $q(n,k,j)$ satisfies recurrence
(\ref{eq:recurrence}) for $k \ge 1$. Indeed, this can 
rewritten as 
\[ q(n,k,j) - q(n-1,k-1,j-1) \ = \ q(n,k-1,j) - 
   q(n-1,k-1,j). \]

We leave it to the reader to verify that: (a) the 
left-hand side equals the number of colored 
permutations $w \in \ZZ_r \wr \fS_{n+1}$ which have 
first coordinate of zero color, $j$ descents, last 
coordinate of zero color and equal to $n+1-k$ and next 
to last coordinate \emph{not} of zero color and equal 
to $n+2-k$; (b) the right-hand side equals the 
number of colored permutations $w \in \ZZ_r \wr 
\fS_{n+1}$ which have first coordinate of zero color, 
$j$ descents, last coordinate of zero color and equal 
to $n+2-k$ and next to last coordinate \emph{not} of 
zero color and equal to $n+1-k$; and (c) swapping the
positions of $n+1-k$ and $n+2-k$ (while preserving the
colors) sets up a bijection between the two sets of 
permutations in (a) and (b).
\end{proof}

\section{Polynomial operators}
\label{sec:operator}

The subdivision operator (see \cite[Section~4]{Bra06} 
\cite[Section~7.3.3]{Bra15} and references therein) is 
a linear operator on polynomials, closely related to 
the barycentric subdivision operation on simplicial 
complexes, which is important in the study of roots 
of real polynomials. This section generalizes this 
operator in the framework of uniform triangulations
and employs this concept as a tool to prove the main 
recurrence and derive new interpretations of the 
coefficients $p_\fF(n,k,j)$. 
\begin{definition} \label{def:sdop} 
Given an $f$-triangle $\fF$ of size $d$, the 
\emph{$\fF$-subdivision operator} is defined as the 
linear operator $\eE_\fF: \RR_d[x] \to \RR_d[x]$ for 
which 
\[ \eE_\fF(x^n) \ := \ \sum_{k=0}^n f_\fF^\circ(k,n) 
   x^k \ = \ f_\fF^\circ(\sigma_n, x) \]
for every $n \in \{0, 1,\dots,d\}$. 
\end{definition}

By Example~\ref{ex:d=2} we have $\eE_\fF(1) = 1$, 
$\eE_\fF(x) = x$ and $\eE_\fF(x^2) = (r-1)x + rx^2$, 
where $r-1$ is the number of interior vertices of any 
$\fF$-uniform triangulation of a 1-simplex. For the
barycentric subdivision we have the explicit formula
$f_\fF^\circ(k,n) = k! S(n,k)$, where $S(n,k)$ is a 
Stirling number of the second kind, and hence 
$\eE_\fF$ coincides with the usual subdivision 
operator \cite[Section~7.3.3]{Bra15}.

Recall that $f_\fF(\Delta, x)$ stands for the 
$f$-polynomial of any $\fF$-uniform triangulation of
$\Delta$. The following proposition generalizes 
\cite[Lemma~7.3.11]{Bra15} and \cite[Lemma~4.3]{Bra06} 
to uniform triangulations.
\begin{proposition} \label{prop:sdop} 
Let $\fF$ be an $f$-triangle of size $d$ and $n \in 
\{0, 1,\dots,d\}$. 

\begin{itemize}
\itemsep=0pt
\item[{\rm (a)}]
We have $f_\fF(\Delta, x) = \eE_\fF (f(\Delta, x))$ 
for every $(n-1)$-dimensional simplicial complex 
$\Delta$. In particular, $\eE_\fF ((x+1)^n) = 
f_\fF(\sigma_n, x)$.

\item[{\rm (b)}]
Suppose that $\fF$ is feasible. Then, $\eE_\fF$ is 
invertible and commutes with the restriction on 
$\RR_d[x]$ of the algebra automorphism $\iI: \RR[x] 
\to \RR[x]$ defined by $\iI(x) = -1-x$.
\end{itemize}
\end{proposition}
\begin{proof}
The coefficient of $x^j$ in $f_\fF 
(\Delta, x)$ is equal to the right-hand side of 
Equation~(\ref{eq:fbasic}); see Remark~\ref{rem:ffF}. 
As a result, 

\begin{eqnarray*}
f_\fF (\Delta, x) & = & \sum_{j=0}^n 
\left( \, \sum_{m=j}^n f_{m-1}(\Delta) \cdot 
f^\circ_\fF(j,m) \right) x^j \ = \ 
\sum_{m=0}^n f_{m-1}(\Delta) \, \sum_{j=0}^m
f^\circ_\fF(j,m) x^j \\ & = & 
\sum_{m=0}^n f_{m-1}(\Delta) \, \eE_\fF(x^m) \ = \ 
\eE_\fF \left( \, \sum_{m=0}^n f_{m-1}(\Delta) x^m 
\right) \ = \ \eE_\fF (f(\Delta, x)).
\end{eqnarray*}

\medskip
\noindent
This verifies the first assertion of part (a). The 
second is the special case $\Delta = \sigma_n$.

For (b), we have to show that $(\eE_\fF \circ \iI)
(x^n) = (\iI \circ \eE_\fF)(x^n)$ for every $n \in 
\{0, 1,\dots,d\}$. By the second assertion of 
Proposition~\ref{prop:sdop} (a) we have $(\eE_\fF \circ 
\iI)(x^n) = \eE_\fF ((-1-x)^n) = (-1)^n \eE_\fF ((x+1)^n)
= (-1)^n f_\fF(\sigma_n, x)$. Since  
$(\iI \circ \eE_\fF)(x^n) = f_\fF^\circ (-1-x)$, the 
desired equality is equivalent to 
Equation~(\ref{eq:ffcirc}) and the proof follows.
Clearly, $\eE_\fF$ is invertible since it has 
triangular form in the standard basis of $\RR_d[x]$
with nonzero diagonal entries $f_\fF^\circ(n,n)$.
\end{proof}

\begin{remark} \label{rem:symmetry} \rm
Just as in the proof of \cite[Lemma~7.3.11]{Bra15} 
for barycentric subdivision, one can use 
Proposition~\ref{prop:sdop} to obtain a new proof 
of part (a) of Corollary~\ref{cor:mainA}. Indeed, 
assume that $h(\Delta,x)$ is symmetric, with center 
of symmetry $n/2$. Then, $(-1)^n f(\Delta,-1-x) = 
f(\Delta,x)$ and applying the operator $\eE_\fF$, 
we get

\begin{eqnarray*}
(-1)^n f_\fF(\Delta,-1-x) & = & (-1)^n (\iI \circ
\eE_\fF) f(\Delta,x) \ = \ 
(-1)^n (\eE_\fF \circ \iI) f(\Delta,x) \\ & = & 
(-1)^n \eE_\fF (f(\Delta,-1-x)) \ = \ 
\eE_\fF ((-1)^n f(\Delta,-1-x)) \\ & = & 
\eE_\fF (f(\Delta,x)) \ = \ f_\fF(\Delta,x).
\end{eqnarray*}

\medskip
\noindent
This means that $h_\fF(\Delta, x)$ is symmetric, with 
center of symmetry $n/2$. 
\qed
\end{remark}

We have shown that $\eE_\fF$ is the linear operator 
which describes the transformation of the 
$f$-polynomial of a simplicial complex under 
$\fF$-uniform triangulation. We now consider the 
corresponding transformation of the $h$-polynomial.
\begin{definition} \label{def:hsdop} 
Given an $f$-triangle $\fF$ of size $d$ and $n \in 
\{0, 1,\dots,d\}$, we define the linear operator 
$\dD_{\fF,n}: \RR_n[x] \to \RR_n[x]$ by setting 
$\dD_{\fF,n}(x^k) := \ p_{\fF,n,k}(x)$ for every 
$k \in \{0, 1,\dots,n\}$. Equivalently, 
\[ \dD_{\fF,n}(h(x)) \ := \ \sum_{k=0}^n h_k 
   p_{\fF,n,k}(x) \]
for every $h(x) = \sum_{k=0}^n h_k x^k \in \RR_n
[x]$. 
\end{definition}

The following proposition lists basic properties 
of $\dD_{\fF,n}$ and interprets the polynomial 
$p_{\fF,n,k}(x)$ by means of the subdivision 
operator $\eE_\fF$.
\begin{proposition} \label{prop:hsdop} 
Let $\fF$ be an $f$-triangle of size $d$ and $n 
\in \{0, 1,\dots,d\}$. 

\begin{itemize}
\itemsep=0pt
\item[{\rm (a)}]
$h_\fF(\Delta, x) = \dD_{\fF,n} (h(\Delta, x))$ for 
every $(n-1)$-dimensional simplicial complex $\Delta$. 

\item[{\rm (b)}]
We have $\dD_{\fF,n} = \jJ^{-1}_n \circ \eE_{\fF,n} 
\circ \jJ_n$, where $\jJ_n: \RR_n[x] \to \RR_n[x]$ is 
the invertible linear operator which maps a 
polynomial $h(x) \in \RR_n[x]$ to the associated
$f$-polynomial and $\eE_{\fF,n}$ is the restriction 
of $\eE_\fF$ to $\RR_n[x]$. 

\item[{\rm (c)}]
The polynomial $p_{\fF,n,k}(x)$ is equal to the 
$h$-polynomial associated to $f_{\fF,n,k}(x) := 
\eE_\fF (x^k(1+x)^{n-k})$, so that 
\[ p_{\fF,n,k}(x) \ = \ (1-x)^n f_{\fF,n,k}
   \left( \frac{x}{1-x} \right), \]
for every $k \in \{0, 1,\dots,n\}$.
\end{itemize}
\end{proposition}
\begin{proof}
Part (a) follows directly from 
Theorem~\ref{thm:mainAB}. Moreover, the proof of 
this theorem shows that given any polynomial $h(x) \in 
\RR_n[x]$ with associated $f$-polynomial $f(x)$,
the $h$-polynomial which corresponds to $\eE_\fF 
(f(x))$ equals $\dD_{\fF,n}(h(x))$. Thus, part (b) 
holds as well. Part (c) follows by applying (b) to
$h(x) = x^k$, whose associated $f$-polynomial 
is $f(x) = x^k(1+x)^{n-k}$.
\end{proof}

\begin{corollary} \label{cor:p-char} 
For any $f$-triangle $\fF$ of size $d$ we have 
\begin{equation}
\label{eq:ppoly-rec}
p_{\fF,n,k}(x) \ = \ p_{\fF,n,k-1}(x) + (x-1) 
p_{\fF,n-1,k-1}(x)
\end{equation}
for $1 \le k \le n \le d$. Equivalently, the 
recurrence~(\ref{eq:recurrence}) holds for the 
coefficients $p_\fF(n,k,j)$.
\end{corollary}
\begin{proof}
Under the notation of 
Proposition~\ref{prop:hsdop} (c),

\begin{eqnarray}
f_{\fF,n,k}(x) & = & \eE_\fF (x^k(1+x)^{n-k}) \ = \ 
\eE_\fF (x^{k+1}(1+x)^{n-k-1} + x^k(1+x)^{n-k-1}) 
\nonumber \\
& = & \eE_\fF (x^{k+1}(1+x)^{n-k-1}) + 
\eE_\fF (x^k(1+x)^{n-k-1}) \nonumber \\ & = & 
f_{\fF,n,k+1}(x) + f_{\fF,n-1,k}(x). 
\label{eq:fpoly-rec}
\end{eqnarray}

\medskip
\noindent
Thus, by the same proposition,

\begin{eqnarray*}
p_{\fF,n,k+1}(x) & = & (1-x)^n f_{\fF,n,k+1}(x/(1-x)) 
\\ & = & (1-x)^n f_{\fF,n,k}(x/(1-x)) - (1-x)^n 
f_{\fF,n-1,k}(x/(1-x)) \\ & = & 
p_{\fF,n,k}(x) - (1-x) p_{\fF,n-1,k}(x)
\end{eqnarray*}

\medskip
\noindent
and the proof follows.
\end{proof}

\begin{example} \label{ex:hsdop-esd} \rm
By Proposition~\ref{prop:hsdop} and 
Equation~(\ref{eq:BW09}), for the $r$-fold edgewise 
subdivision the operator $\dD_{\fF,n}$ has the explicit
form 
\[ \dD_{\fF,n}(h(x)) \ = \ \left(
   (1 + x + x^2 + \cdots + x^{r-1})^n h(x) 
	 \right)^{\langle r,0 \rangle}. \]
For the relation of this operator to the Veronese 
construction for rational formal power series, see 
\cite[Section~3]{BS10}.
\qed
\end{example}

The second interpretation of the polynomials 
$p_{\fF,n,k}(x)$ offered in this section is as 
follows. 
\begin{proposition} \label{prop:relative} 
Let $\fF$ be an $f$-triangle of size $d$ and 
$\Gamma_n$ be an $\fF$-uniform triangulation of the 
simplex $\sigma_n$ for some $n \le d$. Then, $p_{\fF,n,k}
(x) = h(\Gamma_{n,k}, x)$ for every $k \in \{0, 1,\dots,n\}$,
where $\Gamma_{n,k}$ is the relative simplicial complex 
obtained from $\Gamma_n$ by removing all faces carried by 
any $k$ chosen facets of $\sigma_n$. 
\end{proposition}
\begin{proof}
Since $\Gamma_{n,0} = \Gamma_n$, we have 
$h(\Gamma_{n,0}, x) = h(\Gamma_n, x) = h_\fF(\sigma_n, x) 
= p_{\fF,n,0} (x)$ for all $n \le d$. Thus, it suffices to 
show that the polynomials $h(\Gamma_{n,k}, x)$ satisfy 
recurrence (\ref{eq:ppoly-rec}) or, equivalently, that the 
polynomials $f(\Gamma_{n,k}, x)$ satisfy 
recurrence (\ref{eq:fpoly-rec}). Indeed, this is true 
because, by construction, $\Gamma_{n,k+1}$ is the complex 
obtained from $\Gamma_{n,k}$ by removing a relative 
subcomplex combinatorially isomorphic to $\Gamma_{n-1,k}$.
\end{proof}

For example, let $d = 3$ and suppose again that $\fF$ is 
the $f$-triangle for the edgewise subdivision of Figure~1.
By counting directly the faces of the relative simplicial 
complexes $\Gamma_{3,k}$, we find that
\[ f(\Gamma_{3,k}, x) \ = \ \begin{cases}
    1 + 15x + 30x^2 + 16x^3, & \text{if $k=0$} \\
    10x + 26x^2 + 16x^3, & \text{if $k=1$} \\
    6x + 22x^2 + 16x^3, & \text{if $k=2$} \\
    3x + 18x^2 + 16x^3, & \text{if $k=3$}. \end{cases} \]
We leave it to the reader to verify that the 
corresponding $h$-polynomials are exactly those we 
computed as the polynomials $p_{\fF,3,k}(x)$ in 
Example~\ref{ex:d=3}. 

\begin{remark} \label{rem:nonnegativity} \rm
Proposition~\ref{prop:relative} implies the 
nonnegativity of the polynomials $p_{\fF,n,k}(x)$, via 
relative Stanley--Reisner theory, as well as the 
symmetry property $x^n p_{\fF,n,k} (1/x) = 
p_{\fF,n,n-k} (x)$. Indeed, the former follows since 
$\Gamma_{n,k}$ is relatively Cohen--Macaulay, by 
\cite[Corollary~III.7.3~(iii)]{StaCCA}, and hence has 
nonnegative $h$-vector. The latter follows from 
Proposition~\ref{prop:hsymmetry}. 
\end{remark}

\section{Real-rootedness of the $h$-polynomial}
\label{sec:roots}

This section proves an expanded version of 
Theorem~\ref{thm:mainB}.

We recall that a polynomial $g(x) \in \RR[x]$ is
\emph{real-rooted} if all complex roots of $g(x)$ are 
real, or $g(x)$ is the zero polynomial. A real-rooted
polynomial, with roots $\alpha_1 \ge \alpha_2 
\ge \cdots$, is said to \emph{interlace} a 
real-rooted polynomial, with roots $\beta_1 
\ge \beta_2 \ge \cdots$, if
\[ \cdots \le \alpha_2 \le \beta_2 \le \alpha_1 \le
   \beta_1. \]
By convention, the zero polynomial interlaces and is 
interlaced by every real-rooted polynomial and constant 
polynomials interlace all polynomials of degree at most 
one. We refer to \cite[Section~7.8]{Bra15} and the 
references given there for background on the theory of 
interlacing and for any related undefined terminology. 

Every polynomial $g(x) \in \RR_n[x]$ 
can be written uniquely in the form $g(x) = a(x) + 
xb(x)$, where $a(x) \in \RR_n[x]$ is symmetric, with
center of symmetry $n/2$ and $b(x) \in \RR_{n-1}[x]$ 
is symmetric, with center of symmetry $(n-1)/2$.
Following \cite{BS20}, we will say that $g(x)$ has a
\emph{nonnegative, real-rooted symmetric decomposition}
with respect to $n$, if $a(x)$ and $b(x)$ are 
real-rooted polynomials with nonnegative coefficients.
We will also say that such a decomposition is 
\emph{interlacing} if the following equivalent 
(by \cite[Theorem~2.6]{BS20}) conditions hold: (i) 
$a(x)$ interlaces $g(x)$; (ii) $b(x)$ interlaces 
$g(x)$; (iii) $b(x)$ interlaces $a(x)$; and (iv) 
$x^n g(1/x)$ interlaces $g(x)$.

We can now state the main result of this section.
\begin{theorem} \label{thm:mainBB} 
Let $\fF$ be a feasible $f$-triangle of size $d \in 
\NN \cup \{ \infty\}$ and assume the following for some 
$n \in \{1, 2,\dots,d\}$: 

\begin{itemize}
\itemsep=0pt
\item[{\rm (i)}]
$h_\fF(\sigma_m, x)$ is a real-rooted polynomial for 
all $2 \le m < n$.

\item[{\rm (ii)}]
$h_\fF(\sigma_m, x) - h_\fF(\partial \sigma_m, x)$ is 
either identically zero, or a real-rooted polynomial of
degree $m-1$ with nonnegative coefficients which is 
interlaced by $h_\fF(\sigma_{m-1}, x)$, for all 
$2 \le m \le n$.
\end{itemize}
Then:

\begin{itemize}
\itemsep=0pt
\item[{\rm (a)}]
The polynomial $\dD_{\fF,n}(h(x))$ is real-rooted for 
every polynomial $h(x) \in \RR_n[x]$ with nonnegative 
coefficients. In 
particular, $h_\fF(\Delta, x)$ is real-rooted for every
$(n-1)$-dimensional simplicial complex $\Delta$ with 
nonnegative $h$-vector.

\item[{\rm (b)}]
The polynomials $h_\fF(\sigma_n, x)$ and 
$h_\fF(\partial\sigma_n, x)$ are real-rooted and are
interlaced by $h_\fF(\sigma_{n-1}, x)$. Moreover,
$h_\fF(\sigma_n, x)$ has a nonnegative, real-rooted 
and interlacing symmetric decomposition with respect 
to $n-1$.
\end{itemize}
\end{theorem}

Some comments on the hypotheses of the theorem are in
order. We recall the fact \cite[Proposition~3.3]{Wa92}, 
to be used in the sequel, that if two real-rooted 
polynomials with positive leading coefficients 
interlace (respectively, are interlaced by) a nonzero
real-rooted polynomial $g(x)$, then their sum is also
real-rooted and interlaces (respectively, is interlaced 
by) $g(x)$.
\begin{remark} \label{rem:intpoint} \rm
(a) The degree of $h_\fF(\sigma_m, x)$ is at most 
$m-1$ for every positive integer $m$. The last 
sentence of 
Proposition~\ref{prop:hsymmetry} implies that the 
coefficient of $x^{m-1}$ in $h_\fF(\sigma_m, x)$ is
equal to the number $f_\fF^\circ(1,m)$ of interior 
vertices of any $\fF$-uniform triangulation of 
$\sigma_m$. Therefore, the coefficient of $x^{m-1}$ 
in $h_\fF(\sigma_m, x) - h_\fF(\partial \sigma_m, 
x)$ is one less than this number. As a 
consequence, a necessary condition for assumption 
(ii) to be valid is that there exists at least one 
such interior vertex (this is also sufficient for 
$m=2$). No such vertex exists for the 
$r$-fold edgewise subdivision of $\sigma_m$, unless 
$r \ge m$. This matches the range of values of $r$ 
for which \cite[Theorem~1.1]{Jo18} holds.

(b) A sufficient condition for $h_\fF(\sigma_m, x) - 
h_\fF(\partial \sigma_m, x)$ to have nonnegative 
coefficients is that no facet of some $\fF$-uniform 
triangulation of $\sigma_m$ has all vertices on the 
boundary of $\sigma_m$. This is a consequence of 
\cite[Theorem~2.1]{Sta93}.

(c) Assumption (ii) cannot be dropped from the 
hypotheses, even if the assumption that $h_\fF(\sigma_m, 
x)$ is interlaced by $h_\fF(\sigma_{m-1}, x)$ for all 
$m < n$ is added. Indeed, let $d = \infty$ and suppose 
that $\fF$ is the $f$-triangle for 2-fold edgewise 
subdivision. Equation~(\ref{eq:BW09}) implies that
\[ h_\fF(\sigma_n, x) \ = \ \sum_{k \ge 0} 
	 {n \choose 2k} x^k. \]
Setting 
\[ g_n(x) \ := \ \sum_{k \ge 0} {n \choose 2k+1} x^k, 
	 \]
it is well known (see, for instance, 
\cite[Theorem~7.64]{Fi06} \cite[Proposition~3.4]{Jo18})
that these two polynomials are real-rooted and that 
$g_n(x)$ interlaces $h_\fF(\sigma_n, x)$ for every $n$. 
As a result, the latter  
interlaces $xg_n(x)$. Since, by the standard recurrence 
for binomial coefficients, we have $h_\fF(\sigma_{n+1}, 
x) = h_\fF(\sigma_n, x) + xg_n(x)$, it follows that 
$h_\fF(\sigma_n, x)$ interlaces $h_\fF(\sigma_{n+1}, x)$
for every $n$ as well. However, as noted in 
\cite[p.~552]{BW09}, the polynomial $h_\fF(\partial
\sigma_5, x)$ is not real-rooted and hence conclusion 
(a) of Theorem~\ref{thm:mainBB} fails in this case.
\qed
\end{remark}

To prepare for the proof of Theorem~\ref{thm:mainBB},
we focus on the polynomials $p_{\fF,n,k}(x)$. We have
already shown that they have nonnegative coefficients
and that:

\begin{eqnarray}
p_{\fF,n,0}(x) & = & h_\fF(\sigma_n, x),
\label{eq:ppoly-one} \\
p_{\fF,n,k}(x) & = & p_{\fF,n,k-1}(x) + (x-1) 
p_{\fF,n-1,k-1}(x), \ \ \textrm{for} \ k \ge 1, 
\label{eq:ppoly-two} \\
x^n p_{\fF,n,k} (1/x) & = & p_{\fF,n,n-k} (x), 
\label{eq:ppoly-three} \\ & & \nonumber \\
\sum_{k=0}^n p_{\fF,n,k}(x) & = & h_\fF
(\partial \sigma_{n+1}, x). \label{eq:ppoly-four}
\end{eqnarray}

\medskip
Given that $h_\fF(\sigma_n, x)$ has nonnegative 
coefficients and constant term equal to 1 for every 
$n \in \NN$, recurrence~(\ref{eq:ppoly-two}) shows 
that the polynomials $p_{\fF,n,k}(x)$ are nonzero and 
that for fixed $n \in \NN$, their degrees are 
increasing in $k$.

Throughout this section, we set $p_{\fF,n-1,n}(x) := 
h_\fF(\sigma_n, x) - h_\fF(\partial\sigma_n, x)$ for 
$n \ge 1$. This polynomial has zero constant term 
and, as explained in \cite[Remark~4.8]{KMS19}, it 
is always symmetric, with center of symmetry $n/2$ 
(but it does not necessarily have nonnegative 
coefficients).
\begin{lemma} \label{lem:recurrence} 
For every $f$-triangle $\fF$ of size $d$, the 
recurrence
\begin{equation}
\label{eq:prec-long}
p_{\fF,n,k}(x) \ = \ x \sum_{i=0}^{k-1} 
p_{\fF,n-1,i}(x) \, + \, \sum_{i=k}^n p_{\fF,n-1,i}(x) 
\end{equation}
holds for $0 \le k \le n \le d$, where 
$p_{\fF,n-1,n}(x) = h_\fF(\sigma_n, x) - h_\fF(\partial
\sigma_n, x)$.
\end{lemma}
\begin{proof}
For $k \ge 1$, replacing $k$ by $i$ in the 
recurrence~(\ref{eq:ppoly-two}) and summing for $1 \le 
i \le k$, we get
\[ p_{\fF,n,k}(x) \ = \ p_{\fF,n,0}(x) \, + \, x
\sum_{i=0}^{k-1} p_{\fF,n-1,i}(x) \, - \, \sum_{i=0}^{k-1} 
p_{\fF,n-1,i}(x) \]
(which holds trivially for $k=0$).
Equation~(\ref{eq:prec-long}) follows from this equality, 
when combined with Equations~(\ref{eq:ppoly-one}) 
and~(\ref{eq:ppoly-four}).
\end{proof}

A sequence $(g_0(x), g_1(x),\dots,g_m(x))$ of real-rooted
polynomials is called \emph{interlacing} if $g_i(x)$ 
interlaces $g_j(x)$ for all $0 \le i < j \le m$. For that
to happen, it suffices to require that $g_{i-1}(x)$ 
interlaces $g_i(x)$ for all $1 \le i \le m$ and $g_0(x)$ 
interlaces $g_m(x)$ (see \cite[Lemma~2.3]{Bra06} 
\cite[Proposition~3.3]{Wa92}). We are now in position to 
prove Theorem~\ref{thm:mainBB}.

\medskip
\noindent
\emph{Proof of Theorem~\ref{thm:mainBB}.} The proof is 
motivated by \cite[Example~7.8.8]{Bra15} (essentially, the 
special case of barycentric subdivision). Let us consider 
the sequences

\begin{eqnarray*}
\pP_{\fF,n} & := & (p_{\fF,n-1,0}(x), 
               p_{\fF,n-1,1}(x),\dots,p_{\fF,n-1,n}(x)) \\ 
\qQ_{\fF,n} & := & (p_{\fF,n,0}(x), 
                    p_{\fF,n,1}(x),\dots,p_{\fF,n,n}(x)).
\end{eqnarray*}

\medskip
\noindent
By Theorem~\ref{thm:mainAB} and assumption (ii), all 
polynomials in these sequences have nonnegative 
coefficients. We will first show by induction on $n$ 
that these sequences are interlacing (in particular, 
their elements are real-rooted). This is true for 
$n=1$ since $p_{\fF,0,0}(x) = p_{\fF,1,0}(x) = 1$, 
$p_{\fF,0,1}(x) = 0$, $p_{\fF,1,1}(x) = x$,
$p_{\fF,1,1}(x) = (r-2)x$, where $r$ has the same 
meaning as in Example~\ref{ex:d=3}, and $r \ge 2$ by
assumption (ii).

For the inductive step, let us assume that $n \ge 2$ 
and that the sequences $\pP_{\fF,n-1}$ and 
$\qQ_{\fF,n-1}$ are interlacing. We will show that 
so are $\pP_{\fF,n}$ and $\qQ_{\fF,n}$.
As in \cite[Example~7.8.8]{Bra15}, 
Lemma~\ref{lem:recurrence} shows, by an application 
of \cite[Corollary~7.8.7]{Bra15}, that the 
interlacing property for $\pP_{\fF,n}$ implies 
that for $\qQ_{\fF,n}$. Therefore, we only need to 
show that $\pP_{\fF,n}$ is interlacing. Since the 
first $n$ terms of $\pP_{\fF,n}$ form $\qQ_{\fF,n-1}$, 
we already know that $p_{\fF,n-1,i}(x)$ interlaces 
$p_{\fF,n-1,j}(x)$ for $0 \le i \le j \le n-1$. 
Moreover, by assumption (ii), $p_{\fF,n-1,0} (x) = 
h_\fF(\sigma_{n-1},x)$ interlaces $p_{\fF,n-1,n}(x) 
= h_\fF(\sigma_n,x) - h_\fF(\partial\sigma_n, x)$. 
Thus, by the weak transitivity property 
\cite[Lemma~2.3]{Bra06} of interlacing, mentioned 
earlier, it suffices to show that $p_{\fF,n-1,n-1} 
(x)$ interlaces $p_{\fF,n-1,n}(x)$. We may assume 
that the latter is nonzero, in which 
case it must have degree $n-1$ by assumption (ii). 
Since $p_{\fF,n-1,n}(x)$ has zero constant term, 
$p_{\fF,n-1,0} (x)$ is interlaced by $p_{\fF,n-1,n}
(x)/x$. Since both polynomials have degree 
$n-2$ (see Remark~\ref{rem:intpoint} (a)) and 
the latter is symmetric, with center of symmetry 
$(n-2)/2$, the polynomial $x^{n-2} p_{\fF,n-1,0}(1/x)$ 
interlaces $p_{\fF,n-1,n}(x)/x$ and therefore 
$p_{\fF,n-1,n-1} (x) = x^{n-1} p_{\fF,n-1,0}(1/x)$ 
interlaces $p_{\fF,n-1,n}(x)$. This completes the 
induction. 

Suppose now that $h(x) \in \RR[x]$ has nonnegative 
coefficients. Then, by Definition~\ref{def:hsdop}, 
$\dD_{\fF,n}(h(x))$ can be written as a nonnegative 
linear combination of the elements of $\qQ_{\fF,n}$.
Since the latter is interlacing, $\dD_{\fF,n}(h(x))$
is real-rooted \cite[Theorem~7.8.2]{Bra15}. This proves 
part (a). We have already shown that $h_\fF(\sigma_n, 
x) = p_{\fF,n,0} (x)$ is real-rooted. The equality 
\begin{equation}
\label{eq:symdecomp}
h_\fF(\sigma_n, x) \ = \  h_\fF(\partial\sigma_n, 
x) + (h_\fF(\sigma_n, x) - h_\fF(\partial\sigma_n, x)), 
\end{equation}
combined with assumption 
(ii), shows that it also has a nonnegative symmetric
decomposition with respect to $n-1$. Since, as we 
have already shown, $h_\fF(\sigma_n, x) = p_{\fF,n,0} 
(x)$ interlaces 
$p_{\fF,n,n} (x) = x^n h_\fF(\sigma_n, 1/x)$, we 
also have that $h_\fF(\sigma_n, x)$ is interlaced  
by $x^{n-1} h_\fF(\sigma_n, 1/x)$. These facts and 
\cite[Theorem~2.6]{BS20} imply that $h_\fF(\sigma_n, 
x)$ has a nonnegative, real-rooted and interlacing 
symmetric decomposition with respect to $n-1$. This 
proves the last statement of part (b) and that 
$h_\fF(\partial\sigma_n, x)$ is real-rooted as well. 

Finally, recall that
$h_\fF(\sigma_{n-1}, x)$ interlaces all elements of 
$\pP_{\fF,n}$ and $\qQ_{\fF,n-1}$. As a result 
\cite[Proposition~3.3]{Wa92}, it also interlaces the 
sums of the elements of these sequences, which equal 
$h_\fF(\sigma_n, x)$ and $h_\fF(\partial\sigma_n, x)$, 
respectively. This completes the proof of part (b).
\qed

\section{Applications}
\label{sec:apps}

We now discuss how Theorem~\ref{thm:mainB} applies to 
the motivating examples of Section~\ref{sec:examples}.

\begin{example} \label{ex:sd} \rm
Consider the barycentric subdivision and let $\fF$ 
be the corresponding $f$-triangle of infinite size. 
Then, $h_\fF(\sigma_m, x) = h_\fF(\partial\sigma_m, 
x)$ is the $m$th Eulerian polynomial (see
\cite[Section~1.4]{StaEC1}) and thus, assumption (ii)
of Theorem~\ref{thm:mainBB} is satisfied trivially. 
Given Theorem~\ref{thm:mainBB}, a straightforward 
induction on $n$ then shows that $\dD_{\fF,n}(h(x))$ 
is real-rooted for every polynomial $h(x) \in \RR_n
[x]$ with nonnegative coefficients. In particular, 
$h(\sd(\Delta), x)$ is real-rooted for every 
simplicial complex $\Delta$ with nonnegative 
$h$-vector. Hence, Theorem~\ref{thm:mainBB} recovers 
the main result of \cite{BW08} in this case.
\end{example}
\begin{example} \label{ex:esdr} \rm
Let $\fF$ be the $f$-triangle of infinite size which 
corresponds to the $r$-fold edgewise subdivision. 
Setting $c_{n,r}(x) := (1 + x + x^2 + \cdots + 
x^{r-1})^n$, from Equation~(\ref{eq:BW09}) we deduce 
that $h_\fF(\sigma_n, x) = (c_{n,r}(x))^{\langle r,0 
\rangle}$ and 
\[ h_\fF(\partial \sigma_n, x) \ = \ \left( 
c_{n-1,r}(x) (1 + x + x^2 + \cdots + x^{n-1}) 
\right)^{\langle r,0 \rangle}. \]
It is well known \cite[Example~3.76]{Fi06} 
\cite[Proposition~3.4]{Jo18} that the sequence 
$(c_{n,r}(x))^{\langle r,r-i \rangle}_{1 \le i \le r}$
is interlacing. We deduce that, for $r \ge m$,  
\begin{eqnarray*}
h_\fF(\sigma_m, x) - h_\fF(\partial \sigma_m, x) 
& = & \left( c_{m,r}(x) \right)^{\langle r,0 \rangle} 
- \left( c_{m-1,r}(x) (1 + x + x^2 + \cdots + x^{m-1}) 
\right)^{\langle r,0 \rangle} \\ & = & 
\left( c_{m-1,r}(x) (x^m + x^{m+1} + \cdots + x^{r-1})  
\right)^{\langle r,0 \rangle} \\ & = & 
\sum_{i=m}^{r-1} x (c_{m-1,r}(x))^{\langle r,r-i \rangle} 
\end{eqnarray*}
is interlaced by $(c_{m-1,r}(x))^{\langle r,0 \rangle} 
= h_\fF(\sigma_{m-1}, x)$, since the latter is interlaced 
by all of the $(c_{m-1,r}(x))^{\langle r,r-i \rangle}$. 
Hence, all assumptions of Theorem~\ref{thm:mainBB} 
are satisfied when $r \ge n$ and the theorem recovers 
the case $i=0$ of \cite[Theorem 1.1]{Jo18}. 

It was observed in \cite[p.~15]{Ath16b} that 
\begin{equation}
\label{eq:BHMS}
h_\fF(\sigma_{n+1}, x) \ = {\displaystyle
\sum_{w \in \{0, 1\dots,r-1\}^n} x^{\asc(w)}}, 
\end{equation}
where for $w = (w_1, w_2,\dots,w_n) \in \{0,
1\dots,r-1\}^n$, $\asc(w)$ denotes the number of 
indices $i \in \{0, 1,\dots,n-1\}$ such that $w_i < 
w_{i+1}$, with the convention that $w_0 := 0$. Thus, 
part (b) of Theorem~\ref{thm:mainBB} yields the 
following corollary.

\begin{corollary} \label{cor:BHMS} 
For every positive integer $n$, the polynomial on 
the right-hand side of (\ref{eq:BHMS}) has a 
nonnegative, real-rooted and interlacing symmetric 
decomposition with respect to $n$, for every 
$r \ge n+1$.
\qed
\end{corollary}
\end{example}
\begin{example} \label{ex:sdrs} \rm
To illustrate the applicability of 
Theorem~\ref{thm:mainBB}, consider the following
generalization of the two previous examples. Fix  
integers $1 \le s < r$, think of $\Delta$ as a 
geometric simplicial complex and construct a 
triangulation $\sd_{r,s}(\Delta)$ of $\Delta$ as 
follows. First triangulate the $s$-skeleton of $\Delta$ 
with the $r$-fold edgewise subdivision. Then, for $j 
\ge s$ and assuming the triangulation of the 
$j$-skeleton of $\Delta$ has been defined, triangulate 
each $(j+1)$-dimensional face of $\Delta$ by inserting 
one point in the relative interior of that face and 
coning over its boundary, which is already 
triangulated by induction. 

This process defines a triangulation $\sd_{r,s}(\Delta)$ 
of $\Delta$ which reduces to $\sd(\Delta)$ for $s=1$ and 
$r=2$, and to $\esd_r(\Delta)$ when $s$ is at least as 
large as the dimension of $\Delta$. As verified  
in Example~\ref{ex:esdr}, assumption (ii) of 
Theorem~\ref{thm:mainBB} is satisfied for $m \le s$.
Moreover, it is trivially satisfied for $m > s$, since 
then $h_\fF(\sigma_m, x) - h_\fF(\partial\sigma_m, x) = 
0$ (the operation of coning for simplicial complexes 
leaves the $h$-polynomial invariant). A straightforward 
induction on $n$ then shows that all conclusions in the 
statement of Theorem~\ref{thm:mainBB} hold. It would 
be interesting to interpret combinatorially the 
coefficients $p_\fF(n,k,j)$ for this example.
\qed
\end{example}

For the $r$-colored barycentric subdivision (in 
particular, for the interval triangulation), 
the conclusion (a) of Theorem~\ref{thm:mainBB} 
can be deduced from known results. It would still 
be interesting to decide whether assumption (ii)
of the theorem is valid in this case; this question
will be answered in a forthcoming paper.
\begin{proposition} \label{prop:csdr} 
Let $r$ be any positive integer and $\fF$ be the 
$f$-triangle of infinite size for the $r$-colored 
barycentric subdivision. Then, the polynomial 
$\dD_{\fF,n}(h(x))$ is real-rooted for every $h(x) 
\in \RR_n[x]$ with nonnegative coefficients. 

In particular, the $h$-polynomial of the $r$-colored 
barycentric subdivision of $\Delta$ is real-rooted 
for every simplicial complex $\Delta$ with nonnegative 
$h$-vector.
\end{proposition}
\begin{proof}
This is because, by \cite[Theorem~3.1]{BW08}, 
barycentric subdivision transforms any $h$-polynomial 
with nonnegative coefficients to one with only real
nonpositive roots and $r$-fold edgewise subdivision 
(for instance, by \cite[Theorem~4.5.6]{Bre89} or
\cite[Corollary~3.4]{Zh20}) transforms any 
$h$-polynomial with only real nonpositive roots to 
one with the same property. 
\end{proof}

As already mentioned in the proof of 
Proposition~\ref{prop:csdr-coef}, for the $r$-colored 
barycentric subdivision the polynomial 
$h_\fF(\sigma_n, x)$ is known to equal
\[ A^+_{n,r}(x) \ := \sum_{w \in (\ZZ_r \wr \fS_n)^+} 
   x^{\des(w)}, \]
where $(\ZZ_r \wr \fS_n)^+$ is the set of colored 
permutations $w \in \ZZ_r \wr \fS_n$ with first 
coordinate of zero color and $\des(w)$ stands for 
the number of descents of $w \in \ZZ_r \wr \fS_n$.
This polynomial is real-rooted by 
Proposition~\ref{prop:csdr}. The following statement 
partially confirms conclusion (b) of 
Theorem~\ref{thm:mainBB} in this case.
\begin{proposition} \label{prop:binom} 
For all positive integers $n,r$, the polynomial 
$A^+_{n,r}(x)$ has a nonnegative, real-rooted and 
interlacing symmetric decomposition with respect to 
$n-1$.
\end{proposition}
\begin{proof}
Recall that (\ref{eq:symdecomp}) expresses the symmetric 
decomposition of $h_\fF(\sigma_n, x) = A^+_{n,r}(x)$, 
with respect to $n-1$. The polynomial $h_\fF(\partial
\sigma_n, x)$ has nonnegative coefficients, being the 
$h$-polynomial of a triangulation of a sphere, and only
real roots, by Proposition~\ref{prop:csdr}. To show 
that $h_\fF(\sigma_n, x) - h_\fF(\partial\sigma_n, x)$ 
has the same properties, we work as in 
Example~\ref{ex:esdr}. Using the notation adopted 
there, and since the $h$-polynomial of the barycentric
subdivision of $\sigma_n$ or its boundary is given by
the $n$th Eulerian polynomial $A_n(x)$ (see
\cite[Section~1.4]{StaEC1}), we find that
\begin{eqnarray*}
h_\fF(\sigma_n, x) - h_\fF(\partial \sigma_n, x) 
& = & \left( c_{n,r}(x) A_n(x) 
      \right)^{\langle r,0 \rangle} 
- \left( c_{n-1,r}(x) A_n(x) 
      \right)^{\langle r,0 \rangle} \\ & = & 
\left( c_{n-1,r}(x) A_n(x) (x + x^2 + \cdots + x^{r-1})  
\right)^{\langle r,0 \rangle} \\ & = & 
\sum_{i=1}^{r-1} x (g(x))^{\langle r,r-i \rangle},
\end{eqnarray*}
where $g(x) := c_{n-1,r}(x) \, A_n(x) = (1 + x + \cdots 
+ x^{r-1})^{n-1} A_n(x)$. Since $A_n(x)$ is real-rooted,
an application of \cite[Corollary~3.4]{Zh20} shows that 
the sequence $(g(x))^{\langle r,r-i \rangle}_{1 \le i 
\le r}$ is interlacing. Therefore,  
$(h_\fF(\sigma_n, x) - h_\fF(\partial \sigma_n, x))/x$ is 
real-rooted and interlaces $h_\fF(\partial \sigma_n, x) 
= (g(x))^{\langle r,0 \rangle}$ and the proof follows. 
\end{proof}

\section{Further directions}
\label{sec:rem}

Several interesting classes of polynomials, occurring 
in algebraic and geometric combinatorics, were shown 
to have nonnegative real-rooted symmetric 
decompositions in \cite{BS20}. The following statement
exhibits another such class.
\begin{proposition} \label{prop:brasolus} 
The polynomial $h(\sd(\Delta), x)$ has a nonnegative 
real-rooted symmetric decomposition with respect to 
$n-1$ for every triangulation $\Delta$ of the 
$(n-1)$-dimensional ball. 
\end{proposition}
\begin{proof}
We recall Equation~(\ref{eq:BW08}).
Shifting $m$ to $m-1$ and replacing $i$ in the inner 
summation with $n-i$ first, and then replacing $m$ 
with $-m$ and applying 
\cite[Proposition~4.2.3]{StaEC1}, this equation 
may be rewritten successively as
\begin{eqnarray*} 
\sum_{m \ge 1} \left( \, \sum_{i=0}^n h_{n-i}(\Delta) 
m^i(m-1)^{n-i} \right) x^m & = & 
\frac{x h(\sd(\Delta), x)}{(1-x)^{n+1}} \ \ \
\Leftrightarrow \\
\sum_{m \ge 1} \left( \, \sum_{i=0}^n h_{n-i}(\Delta) 
m^i(m+1)^{n-i} \right) x^m & = & 
\frac{x^n h(\sd(\Delta), 1/x)}{(1-x)^{n+1}} \
\Leftrightarrow \\
\sum_{m \ge 0} \left( \, \sum_{i=0}^n h_{n-i}(\Delta) 
m^i(m+1)^{n-i} \right) x^m & = & 
\frac{x^n h(\sd(\Delta), 1/x)}{(1-x)^{n+1}} \, . \\
\end{eqnarray*}
The last equivalence holds because $h_n(\Delta) = 0$. 
As a consequence of \cite[Lemma~2.3]{Sta93}, we have
\[ h_n(\Delta) + h_{n-1} (\Delta) + \cdots + h_{n-i}
   (\Delta) \ \le \ h_0(\Delta) + h_1 (\Delta) + \cdots 
	 + h_i(\Delta) \]
for all $0 \le i \le n/2$. An application of 
\cite[Theorem~2.13]{BS20} shows that $x^n h(\sd(\Delta), 
1/x)$ has a nonnegative real-rooted symmetric 
decomposition with respect to $n$. Since $x^n 
h(\sd(\Delta), 1/x)$ has zero constant term, this is 
equivalent to the desired statement.
\end{proof}

It seems natural to ask for conditions on an 
$f$-triangle $\fF$ which guarantee that $h_\fF(\Delta, 
x)$ has a nonnegative, real-rooted symmetric decomposition 
for reasonably broad classes of simplicial complexes 
$\Delta$. A forthcoming paper will show that, in fact, 
this is the case under the same conditions on $\fF$, as 
those of Theorem~\ref{thm:mainBB}, and the same conditions 
on $h(\Delta, x)$, as those in the special case of 
barycentric subdivision \cite[Theorem~2.13]{BS20}. 
It may also be interesting to study the 
`deranged map' of \cite[Section~3.2]{BS20} in the 
framework of uniform triangulations.

\bigskip
\noindent \textbf{Acknowledgments}. This article was 
written while the author was in residence at the Institute  
Mittag-Leffler in Djursholm, Sweden during the `Algebraic 
and Enumerative Combinatorics' program in Spring 2020, 
organized by Sara Billey, Peter Br\"and\'en, Sylvie 
Corteel and Svante Linusson and supported by the Swedish 
Research Council under grant no. 2016-06596. The author 
wishes to thank Katharina Jochemko and Martina 
Juhnke-Kubitzke for useful discussions and comments, 
respectively.


\begin{thebibliography}{99}
%
\bibitem{AW18}
S.~Ahmad and V.~Welker,
\textit{On partial barycentric subdivision},
Results Math. {\bf~73} (2018), Article 21, 20pp.
%
\bibitem{AN20}
I.~Anwar and S.~Nazir,
\textit{The $f$- and $h$-vectors of interval 
subdivisions},
J. Combin. Theory Series A {\bf~169} (2020), 
Article 105124, 22pp.
%
\bibitem{Ath14}
C.A.~Athanasiadis,
\textit{Edgewise subdivisions, local $h$-polynomials 
and excedances in the wreath product 
${\mathbb Z}_r \wr \mathfrak{S}_n$},
SIAM J. Discrete Math. {\bf~28} (2014), 1479--1492.
%
\bibitem{Ath16a}
C.A.~Athanasiadis,
\textit{A survey of subdivisions and local $h$-vectors},
in \textit{The Mathematical Legacy of Richard~P.~Stanley}
(P.~Hersh, T.~Lam, P.~Pylyavskyy, V.~Reiner, eds.),
Amer. Math. Society, Providence, RI, 2016, pp.~39--51.
%
\bibitem{Ath16b}
C.A.~Athanasiadis,
\textit{The local $h$-polynomial of the edgewise 
subdivision of the simplex},
Bull. Hellenic Math. Soc. (N.S.) {\bf~60} (2016), 
11--19.
%
\bibitem{Ath18}
C.A.~Athanasiadis,
\textit{Gamma-positivity in combinatorics and geometry},
S\'em. Lothar. Combin. {\bf~77} (2018), Article B77i, 
64pp (electronic).
%
\bibitem{Ath20a}
C.A.~Athanasiadis,
\textit{Some applications of Rees products of posets to 
equivariant gamma-positivity},
Algebr. Comb. {\bf~3} (2020), 291--300.
%
\bibitem{Ath20}
C.A.~Athanasiadis,
\textit{Binomial Eulerian polynomials for colored 
permutations},
J. Combin. Theory Series A {\bf~173} (2020), 
Article 105214, 38pp.
%
\bibitem{ABK20+}
C.A.~Athanasiadis, J-M.~Brunink and M.~Juhnke-Kubitzke,
\textit{Combinatorics of antiprism triangulations},
{\tt arXiv:2006.10789}.
%
\bibitem{BS10} 
M.~Beck and A.~Stapledon, 
\textit{On the log-concavity of Hilbert series of Veronese 
subrings and Ehrhart series}, 
Math. Z. {\bf~264} (2010), 195--207.
%
%
\bibitem{Bra06}
P.~Br\"and\'en,
\textit{On linear transformations preserving the P\'olya 
frequency property},
Trans. Amer. Math. Soc. {\bf~358} (2006), 3697--3716.
%
\bibitem{Bra15}
P.~Br\"and\'en,
\textit{Unimodality, log-concavity, real-rootedness 
and beyond},
in \textit{Handbook of Combinatorics} (M.~Bona, ed.), 
CRC Press, 2015, pp.~437--483.
%
\bibitem{BS20}
P.~Br\"and\'en and L.~Solus,
\textit{Symmetric decompositions and real-rootedness},
Int. Math. Res. Not. (to appear).
%
\bibitem{Bre89}
F.~Brenti,
\textit{Unimodal, log-concave and P\'olya frequency 
sequences in combinatorics},
Mem. Amer. Math. Soc., no. 413, 1989.
%
\bibitem{BW08} 
F.~Brenti and V.~Welker, 
\textit{$f$-vectors of barycentric subdivisions}, 
Math. Z. {\bf~259} (2008), 849--865.
%
\bibitem{BW09}
F.~Brenti and V.~Welker,
\textit{The Veronese construction for formal power 
series and graded algebras},
Adv. in Appl. Math. {\bf~42} (2009), 545--556.
%
\bibitem{DPS12}
E.~Delucchi, A.~Pixton and L.~Sabalka,
\textit{Face vectors of subdivided simplicial complexes},
Discrete Math. {\bf~312} (2012), 248--257.
%
\bibitem{Fi06}
S.~Fisk,
\textit{Polynomials, roots, and interlacing},
{\tt arXiv:0612833}.
%
\bibitem{HS20+}
M.~Hlavacex and L.~Solus,
\textit{Subdivisions of shellable complexes},
{\tt arXiv:2003.07328}.
%
\bibitem{IJ03}
I.~Izmestiev and M.~Joswig,
\textit{Branching coverings, triangulations and 
3-manifolds},
Adv. Geom. {\bf~3} (2003), 191--225.
%
\bibitem{Jo18}
K.~Jochemko,
\textit{On the real-rootedness of the Veronese construction 
for rational formal power series},
Int. Math. Res. Not. {\bf~2018} (2018), 4780--4798.
%
\bibitem{KMS19}
M.~Juhnke-Kubitzke, S.~Murai and R.~Sieg,
\textit{Local $h$-vectors of quasi-geometric and 
barycentric subdivisions},
Discrete Comput. Geom. {\bf~61} (2019), 364--379.
%
\bibitem{Sav13}
C.~Savvidou,
\textit{Barycentric subdivisions, clusters and 
permutation enumeration} (in Greek),
Doctoral Dissertation, University of Athens, 2013.
%
\bibitem{Sta87}
R.P.~Stanley,
\textit{Generalized $h$-vectors, intersection 
cohomology of toric varieties and related results},
in \textit{Commutative Algebra and Combinatorics}
(N.~Nagata and H.~Matsumura, eds.),
Adv. Stud. Pure Math. {\bf~11}, Kinokuniya, 
Tokyo and North-Holland,
Amsterdam-New York, 1987, pp.~187--213.
%
\bibitem{Sta92}
R.P.~Stanley,
\textit{Subdivisions and local $h$-vectors},
J. Amer. Math. Soc. {\bf~5} (1992), 805--851.
%
\bibitem{Sta93}
R.P.~Stanley,
\textit{A monotonicity property of $h$-vectors 
and $h^*$-vectors},
European J. Combin. {\bf~14} (1993), 251--258.
%
\bibitem{StaCCA}
R.P.~Stanley,
Combinatorics and Commutative Algebra,
second edition, Birkh\"auser, Basel, 1996.
%
\bibitem{StaEC1}
R.P.~Stanley,
Enumerative Combinatorics, vol.~1,
Cambridge Studies in Advanced Mathematics {\bf~49},
Cambridge University Press, second edition, 
Cambridge, 2011.
%
\bibitem{Wa92}
D.G.~Wagner,
\textit{Total positivity of Hadamard products},
J. Math. Anal. Appl. {\bf~163} (1992), 459--483.
%
\bibitem{Zh20}
P.B.~Zhang,
\textit{Interlacing polynomials and the Veronese 
construction for rational formal power series},
Proc. Roy. Soc. Edinburgh Sect. A (to appear).
%
\end{thebibliography}
\end{document}